\documentclass[12pt]{article}
\usepackage{delarray} 
\usepackage{amscd}
\usepackage{amssymb}
\usepackage{amsmath}
\usepackage{amsthm,amsxtra}
\usepackage{latexsym,amsmath,mathrsfs}
\usepackage[bookmarks,colorlinks]{hyperref}
\usepackage[all]{xy}
\usepackage{amsfonts}
\usepackage{color}
\usepackage{fancybox}
\usepackage{colordvi}
\usepackage{multicol}
\usepackage{textcomp}
\usepackage{stmaryrd}
\usepackage[dvips]{graphicx}

\usepackage{enumerate,xspace}
\usepackage{geometry}
\usepackage{tocbibind}

\newcommand{\cris}{\mathrm{cris}}
\newcommand{\Fil}{{\mathrm{Fil}}}
\def\st{\mathrm{st}}
\def\dR{\mathrm{dR}}
\def\rig{\mathrm{rig}}
\def\cyc{\mathrm{cyc}}

\def\gr{\mathrm{gr}}

\def\Sh{\mathrm{Sh}}
\def\GL{\mathrm{GL}}

\def\Max{\mathrm{Max}}

\def\BX{\mathbf{X}}
\def\BB{\mathbf{B}}
\def\BZ{\mathbf{Z}}
\def\BV{\mathbf{V}}
\def\BD{\mathbf{D}}
\def\BQ{\mathbf{Q}}
\def\BN{\mathbf{N}}
\def\BF{\mathbf{F}}

\def\BC{\mathbf{C}}

\def\CL{\mathcal{L}}

\def\CR{\mathscr{R}}
\def\CD{\mathcal{D}}
\def\CV{\mathcal{V}}
\def\CM{\mathcal{M}}

\def\CF{\mathcal{F}}

    \theoremstyle{plain}

    \newtheorem{thm}{Theorem}[section]
    \newtheorem{cor}[thm]{Corollary}
    \newtheorem{lem}[thm]{Lemma}
    
    \newtheorem{prop}[thm]{Proposition}

    \theoremstyle{definition}

    \newtheorem{defn}[thm]{Definition}

    \theoremstyle{remark}
    \newtheorem {rem}[thm]{Remark}

    \numberwithin{equation}{section}

\begin{document}

\title{Derivatives of Frobenius and Derivatives of Hodge--Tate weights}
\author{Bingyong Xie \\ \small Department of Mathematics, East China Normal University, Shanghai, China \\ \small byxie@math.ecnu.edu.cn}

\date{}
\maketitle

\begin{abstract} In this paper we study the derivatives of Frobenius
and the derivatives of Hodge--Tate weights for families of Galois
representations with triangulations. We generalize the
Fontaine--Mazur $\CL$-invariant and use it to build a formula which
is a generalization of the Colmez--Greenberg--Stevens formula. For
the purpose of proving this formula we show two auxiliary results
called projection vanishing property and ``projection vanishing
implying $\CL$-invariants'' property.
\end{abstract}

\section*{Introduction}

In the remarkable paper \cite{GS} Greenberg and Stevens proved a
formula conjectured by Mazur, Tate, and Teitelbaum \cite{MTT}, for
the derivative at $s=1$ of the $p$-adic $L$-function of an elliptic
curve $E$ over $\BQ$ when $p$ is a prime of split multiplicative
reduction. An important quantity in this formula is the so called
$\CL$-invariant, namely $\CL(E)=\log_p(q_E)/v_p(q_E)$ where $q_E\in
p \BZ_p$ is the Tate period for $E$. A key ingredient of their proof
is Hida's families of $p$-adic ordinary Hecke eigenforms. For the
weight $2$ newform $g$ attached to $E$, there exists one such family
containing $g$. Let $\alpha$ be the function of $U_p$-eigenvalues of
the forms in this family. On the one hand, by considering the
infinitesimal deformation of the Galois representation in this Hida
family, they proved
\begin{equation}
\CL(E)= -2 \frac{\alpha'(g)}{\alpha(g)}. \label{eq:L-elliptic}
\end{equation}
On the other hand, they used the family to construct a two variable
$p$-adic $L$-function. The two variable $p$-adic $L$-function was
used to prove the following formula
\begin{equation} \label{eq:der}
\frac{L'_p(g, 1)}{L(g,1)} = -2\frac{\alpha'(g)}{\alpha(g)} .
\end{equation} Combining (\ref{eq:L-elliptic}) and (\ref{eq:der})
proves the conjectural formula.

In this paper, we focus on the equality (\ref{eq:L-elliptic}), which
was generalized by Colmez \cite{Cz2010} to the non-ordinary setting.
We state Colmez's result below.

Let $S$ be an affiniod $E$-algebra,  $\CV$ a $2$-dimensional free
$S$-representation of
$G_{\BQ_p}=\mathrm{Gal}(\overline{\BQ}_p/\BQ_p)$. Let $\{v_1, v_2\}$
be a basis of $\CV$ over $S$. Let $\sigma\mapsto A_\sigma$ be the
matrix of $\sigma\in G_{\BQ_p}$ with respect to this basis. Then
there exist $\delta, \kappa\in S$ such that
$$\log(\det A_\sigma) = \delta \psi_1(\sigma) + \kappa
\psi_2(\sigma)$$ for every $\sigma\in G_{\BQ_p}$. Here, $\psi_1:
G_{\BQ_p}\rightarrow E$ is the unramified additive character of
$G_{\BQ_p}$ such that $\psi_1(\sigma)=1$ if $\sigma$ induces the
Frobenius $x \mapsto x^p$ on $\overline{\BF}_p$;
$\psi_2:G_{\BQ_p}\rightarrow E$ is the additive character that is
the logarithm of the cyclotomic character $\chi_\cyc$.

\begin{thm}\label{thm:colmez} $($\cite{Cz2010}$)$ Suppose that, at each closed point $z$ of $\mathrm{Max}(S)$ one of the Hodge--Tate weight of $\CV_z$ is $0$, and there exists $\alpha\in S$
such that $(\BB_{\cris,S}^{\varphi=\alpha}\widehat{\otimes}_S \CV
)^{G_{\BQ_p}}$ is locally free of rank $1$ over $S$. Suppose $z_0$
is a closed point of $\mathrm{Max}(S)$ such that $\CV_{z_0}$ is
semistable with Hodge--Tate weights \footnote{In this paper, the
Hodge--Tate weights are defined to be minus the generalized
eigenvalues of Sen's operators. In particular the Hodge--Tate weight
of the cyclotomic character $\chi_\cyc$ is $-1$.} $0$ and $k\geq 1$.
Then the differential
$$\frac{\mathrm{d}\alpha}{\alpha} - \frac{1}{2} \CL \mathrm{d}\kappa+
\frac{1}{2}\mathrm{d}\delta $$ is zero at $z_0$, where $\CL$ is the
Fontaine--Mazur $\CL$-invariant of $\CV_{z_0}$.
\end{thm}

The condition that
``$(\BB_{\cris,S}^{\varphi=\alpha}\widehat{\otimes}_S \CV
)^{G_{\BQ_p}}$ is locally free of rank $1$ over $S$'' in Theorem
\ref{thm:colmez} is equivalent to that $\CV$ admits a triangulation.
Roughly Theorem \ref{thm:colmez} says that the derivatives of
Frobenius and the derivatives of Hodge--Tate weights of a family of
$2$-dimensional representations of $G_{\BQ_p}$ with a triangulation
satisfy a non-trivial relation at each semistable (but
non-crystalline) point.

There have been several generalizations of Colmez's theorem:

(1) In \cite{Pott}  Pottharst generalized Colmez's theorem to
families of not necessarily \'etale $(\varphi,\Gamma)$-modules of
rank $2$ with a triangulation.


(2) Zhang \cite{Zhang} agian considered families of $2$-dimensional
Galois representations but replaced the base field $\BQ_p$ by any
finite extension of $\BQ_p$. In \cite{Ding, Ding-2} Ding applied
Zhang's work to the $p$-adic Langlands program for $\mathrm{GL}_2$.

(3) Ding \cite{Ding-2} generalized Zhang's work to a more general
setting by establishing a Colmez--Greenberg--Stevens formula for
partially de Rham families.

(4) For other generalizations see \cite{Ben2010, Ben}.

For some historical account of notions of $\mathcal{L}$-invariant
especially for modular forms of higher weights, and comparisons
among them, the readers are invited to consult Colmez's paper
\cite{Cz2005}.

In this paper we generalize Theorem \ref{thm:colmez} to the case of
$n$-dimensional representations of $G_{\BQ_p}$ ($n\geq 2$).
Precisely, what we concern is the relation at a special point among
derivatives of Frobenius and derivatives of Hodge--Tate weights of a
family of $n$-dimensional representations of $G_{\BQ_p}$ with a
triangulation. Our motivation is to generalize Ding's method
\cite{Ding, Ding-2} to higher dimensional Galois representations.

Let us explain what a triangulation is. Let $E$ be a finite
extension of $\BQ_p$. Given an $n$-dimensional $E$-representation
$V$ of $G_{\BQ_p}$, by the theory of Fontaine and the works of
Colmez--Chenbonnier and Berger, it is attached to a
$(\varphi,\Gamma)$-module $\mathrm{D}_{\mathrm{rig}}(V)$ of rank $n$
over the Robba ring $\CR_E$. Here, $\CR_E$ is a ring of power series
with coefficients in $E$. A key observation of Colmez is that
although the Galois representation $V$ may be irreducible,
$\mathrm{D}_{\mathrm{rig}}(V)$ may be reduced, and very often, it is
the successive extension of rank one $(\varphi,\Gamma)$-modules. In
the latter case, we call $V$ admits a triangulation. More precisely,
there is a filtration $\Fil_\bullet M$ on
$M=\mathrm{D}_{\mathrm{rig}}(V)$ consisting of saturated
$(\varphi,\Gamma)$-submodules of $M$ with
$\mathrm{rank}_{\CR_E}\Fil_iM =i$ such that $\Fil_iM/\Fil_{i-1}M$
($1\leq i\leq n$) is of rank $1$, i.e. of the form $\CR_E(\delta_i)$
where $\delta_i$ is an $E^\times$-valued character of
$\BQ_p^\times$. We call $(\delta_1, \dots, \delta_n)$ the
triangulation data for $V$ or $\mathrm{D}_{\mathrm{rig}}(V)$. When
$V$ is semistable, $-w_{\delta_1}, \dots, -w_{\delta_n}$ coincide
with the Hodge--Tate weights of $V$, and
$\delta_1(p)p^{w_{\delta_1}}, \dots, \delta_n(p)p^{w_{\delta_n}}$
coincide with eigenvalues of Frobenius of $V$. Here for a character
$\delta$ of $\BQ_p^\times$, $w_\delta$ is the weight of $\delta$
whose definition is given in Section \ref{sec:tri-ref}.

The significance of triangulations has been confirmed by many works.
It was applied to Fontaine--Mazur conjecture in Kisin's work
\cite{Kisin}, and to  $p$-adic Langlands correspondence in
 Colmez's work \cite{Col-ast}. Chenevier \cite{Chen} used it as a tool to show that
crystalline representations are Zariski-dense in many irreducible
components of $\mathfrak{X}_d$, where $\mathfrak{X}_d$ is the
$p$-adic analytic space classifying the semisimple continuous
representations $G_{\BQ_p}\rightarrow
\mathrm{GL}_d(\overline{\BQ}_p)$. Triangulation was used to
determine the spaces of locally analytic vectors of the unitary
principal series of $\GL_2(\BQ_p)$ in the papers \cite{Col2014} and
\cite{LXZ}.

Before stating our main theorem we introduce a generalization of the
Fontaine--Mazur $\CL$-invariant. Let $(D, \varphi, N, \Fil^\bullet)$
be the admissible filtered $E$-$(\varphi,N)$-module attached to $V$,
$\CF$ the refinement on $D$ corresponding to the triangulation on
$M$. Throughout this paper we assume that {\it $\varphi$ is
semisimple on $D$.} The monodromy $N$ induces an operator $N_\CF$ on
the grading module
$$\mathrm{gr}^\CF_\bullet D=\bigoplus_{i=1}^{\dim_ED}\CF_i
D/\CF_{i-1}D.$$ If $s, t\in \{1, \dots, \dim_E D\}$ satisfy $s<t$
and $N_\CF (\mathrm{gr}^\CF_t D)=\mathrm{gr}^\CF_s D$, then we say
that $s$ is marked for $\CF$ and write $t=t_\CF(s)$. Whether $s$ is
marked does not depend on $\varphi$ and $\Fil^\bullet$. We will
introduce another notion ``being strongly marked'' (see Definition
\ref{defn:Fontaine-Mazur}) which depends not only on $N$ and $\CF$
but also on $\varphi$ and $\Fil^\bullet$. If $s$ is strongly marked,
we can attach to $s$ an invariant denoted by $\CL_{\CF,s}$.

Now we can state our main theorem as follows.

\begin{thm}\label{thm:main} Let $S$ be an affinoid $E$-algebra. Let $\CV$ be an $S$-representation of
$G_{\BQ_p}$ with a triangulation and the associated triangulation
date $(\delta_1, \dots, \delta_n)$. Let $z_0$ be a closed point of
$\mathrm{Max}(S)$, $E_{z_0}$ the residue field of $S$ at $z_0$.
Suppose that $\CV_{z_0}$ is semistable and $\varphi$ is semisimple
on $D$, where $D$ is the filtered $E_{z_0}$-$(\varphi,N)$-module
attached to $\CV_{z_0}$. Let $\CF$ be the refinement on $D$
corresponding to the triangulation of $\CV_{z_0}$. Suppose that
$s\in \{1, \dots, n-1\}$ is strongly marked for $\CF$, $t=t_\CF(s)$.
Then
$$ \frac{\mathrm{d}\delta_t(p)}{\delta_t(p)} - \frac{\mathrm{d}\delta_s(p)}{\delta_s(p)} + \CL_{\CF,s} ( \mathrm{d} w_{\delta_t} -\mathrm{d} w_{\delta_s} ) $$
is zero at $z_0$.
\end{thm}

We remark that, when $s$ is marked for $\CF$ and $t_\CF(s)=s+1$, $s$
is strongly marked for $\CF$ if and only if $w_{\delta_{s}, z_0}
> w_{\delta_{s+1,z_0}}$.

An especially interesting case is when the rank of the monodromy $N$
of $D$ is equal to $\dim_E D-1$. Let $e_n$ be an element not in
$N(D)$ such that $\varphi (e_n)\in E e_n$. For $i=1,\dots, n-1$ put
$e_i=N^{n-i}e_n$. Then $D$ admits a unique triangulation $\CF$ and
$\CF_i D= Ee_1\oplus \cdots \oplus Ee_i$ for all $i=1,\dots, n$.
Write $k_i=-w_{\delta_{i,z_0}}$. Then $k_1, \dots, k_n$ are
Hodge--Tate weights of $\CV_{z_0}$. There always exists an
upper-triangular matrix $(\ell_{j,i})_{n\times n}$ such that $\{e_i
+ \sum\limits_{1\leq j<i} \ell_{j,i}e_j: i=1, \dots, n\}$ is an
$E$-basis of $D$ compatible with the Hodge filtration.

\begin{thm} \label{thm:main-b} With the above notations suppose that
$k_1< k_2<\cdots< k_n$. Then $$
\frac{\mathrm{d}\delta_{s+1}(p)}{\delta_{s+1}(p)} -
\frac{\mathrm{d}\delta_s(p)}{\delta_s(p)} + \ell_{s, s+1} (
\mathrm{d} w_{ \delta_{s+1} } -\mathrm{d} w_{\delta_{s}} ) $$ is
zero at $z_0$.
\end{thm}

When $n=2$, the condition $k_1<k_2$ automatically holds. So Theorem
\ref{thm:main-b} covers Theorem \ref{thm:colmez}. Indeed, under the
condition of Theorem \ref{thm:colmez} we have
$\mathrm{d}w_{\delta_1}=0$,
$\frac{\mathrm{d}\alpha}{\alpha}=\frac{\mathrm{d}\delta_1(p)}{\delta_1(p)}$,
$\mathrm{d}\delta=
-\frac{\mathrm{d}\delta_1(p)}{\delta_1(p)}-\frac{\mathrm{d}\delta_2(p)}{\delta_2(p)}$
and $\mathrm{d}\kappa=\mathrm{d} w_{\delta_2}$.

There are two directions to generalize Theorem \ref{thm:main}. One
is to consider families of (not necessarily \'etale)
$(\varphi,\Gamma)$-modules instead of families of Galois
representations. The other is that the base field $\BQ_p$ is
replaced by a finite extension of $\BQ_p$. We plan to address these
generalizations in a future work.

There are two potential applications of Theorem \ref{thm:main}. One
is to the exceptional zero phenomenon, and the other is to the
$p$-adic Langlands program. In the case of $n=2$ the former is done
in \cite{GS} and the latter is done in \cite{Ding, Ding-2}. In
Section \ref{sec:apply} we provide some discussion on the latter
application.

We sketch the proof of Theorem \ref{thm:main}.

From $\CV$ we obtain an infinitesimal deformation of $\CV_{z_0}$ and
attach to this infinitesimal deformation a $1$-cocycle $c:
G_{\BQ_p}\rightarrow\CV_{z_0}^*\otimes_{E_{z_0}} \CV_{z_0}$. Let
$\{e_1, \dots, e_n\}$ be a basis of $D$ that is $s$-perfect for
$\CF$, $\{e^*_1, \dots e^*_n\}$ the dual basis of $\{e_1, \dots,
e_n\}$. (See Definition \ref{defn:s-perfect-basis} for the precise
meaning of $s$-perfect basis.) Let $\pi_{h,\ell}$ be the composition
of the inclusion
$$ \CV_{z_0}^*\otimes_{E_{z_0}} \CV_{z_0}\hookrightarrow
\BB_{\st,E_{z_0}}\otimes_{E_{z_0}} (\CV_{z_0}^*\otimes_{E_{z_0}}
\CV_{z_0}) $$ and the projection
$$ \BB_{\st,E_{z_0}}\otimes_{E_{z_0}} (\CV_{z_0}^*\otimes_{E_{z_0}}
\CV_{z_0})\rightarrow \BB_{\st,E_{z_0}}, \hskip 10pt \sum_{i,j}
b_{ij}e^*_j\otimes e_i \mapsto b_{\ell h}. $$

We have the following projection vanishing property (Theorem
\ref{thm:main-use-a}) and ``projection vanishing implying
$\CL$-invariant'' property (Theorem \ref{thm:main-use-b}).

\begin{thm} \label{thm:main-use-a} Suppose that $\varphi$ is semisimple on $D$. Let $c:G_{\BQ_p}\rightarrow \CV_{z_0}^*\otimes_{E_{z_0}}\CV_{z_0}$ be a
$1$-cocycle coming from an infinitesimal deformation of $\CV_{z_0}$.
If $h<\ell$, then $\pi_{h,\ell}([c])=0$ in $H^1(\BB_{\st,E_{z_0}})$.
\end{thm}

\begin{thm}\label{thm:main-use-b} Suppose that $\varphi$ is semisimple on $D$. Let $c$ be a $1$-cocycle $G_{\BQ_p}\rightarrow \CV_{z_0}^*\otimes_{E_{z_0}}
\CV_{z_0}$ satisfying the projection vanishing property. If $s$ is
strongly marked for $\CF$ and $t=t_\CF(s)$, then there exist
$\gamma_{s,1}, \gamma_{s,2}, \gamma_{t,1}, \gamma_{t,2}\in E_{z_0}$
and $x_s,x_t\in \BB_{\st,E_{z_0}}^{\varphi=1}$ such that
$$ \pi_{i,i}(c_\sigma) = \gamma_{i,1} \psi_1+ \gamma_{i,2}\psi_2+ (\sigma-1)x_i , \hskip 10pt i=s,t
.$$ Furthermore
$\gamma_{s,1}-\gamma_{t,1}=\CL_{\CF,s}(\gamma_{s,2}-\gamma_{t,2})$.
\end{thm}

Theorem \ref{thm:main} follows from Theorem \ref{thm:main-use-a},
Theorem \ref{thm:main-use-b} and a computation relating
$\gamma_{i,1}$, $\gamma_{i,2}$ to
$\frac{\mathrm{d}\delta_i(p)}{\delta_i(p)}$ and
$\mathrm{d}w_{\delta_i}$.

Our paper is organized as follows. In Section \ref{sec:gal-coh} we
provide preliminary results on Galois cohomology. The proof of the
``projection vanishing implying $\CL$-invariant'' property needs the
functors $\BX_\st$ and $\BX_\dR$ used in \cite{CzFon} where they are
denoted by $V^0_{st}$ and $V^1_{st}$ respectively. In Section
\ref{sec:gal-rep} we give a systematic study on these two functors.
The relation between triangulations and refinements is reviewed in
Section \ref{sec:tri-ref}. In Section \ref{sec:L-inv} we introduce
the concepts of being marked and strongly marked, and define
$\CL$-invariants. The ``projection vanishing implying
$\CL$-invariant'' property is proved in Section \ref{sec:main-teck},
and the projection vanishing property is proved in Section
\ref{sec:aux-res}. In section \ref{sec:proof-main} we combine
results in Section \ref{sec:main-teck} and Section \ref{sec:aux-res}
to prove Theorem \ref{thm:main}. Finally in Section \ref{sec:apply}
we discuss an application of Theorem \ref{thm:main} in $p$-adic
Langlands program.

\section*{Notation}

For a $G_{\BQ_p}$-module $M$ we write $H^i(M)$ for the cohomology
group $H^i(G_{\BQ_p}, M)$. For a $1$-cocycle $c:G_{\BQ_p}\rightarrow
M$ let $[c]$ denote the class of $c$ in $H^1(M)$. For a
$G_{\BQ_p}$-module $M$ let $M(i)$ denote the twist of $M$ by
$\chi_\cyc^i$, where $\chi_\cyc$ is the cyclotomic character.

Let $E$ be a finite extension of $\BQ_p$ considered as a base field
with trivial action of $G_{\BQ_p}$. Let $\psi_1:
G_{\BQ_p}\rightarrow E$ be the unramified additive character of
$G_{\BQ_p}$ such that $\psi_1(\sigma)=1$ if $\sigma$ induces the
Frobenius $x \mapsto x^p$ on $\overline{\BF}_p$. Let
$\psi_2:G_{\BQ_p}\rightarrow E$ be the additive character that is
the logarithm of $\chi_\cyc$. Then $[\psi_1]$ and $[\psi_2]$ form a
basis of $H^1(E)=\mathrm{Hom}(G_{\BQ_p},E)$ over $E$.

For an affinoid $E$-algebra $S$ and a closed point $z\in \Max(S)$,
let $E_z$ denote the residue field of $S$ at $z$. For an $S$-module
$\CM$ we put $\CM_z=\CM\otimes_S E_z$.

Let $\BN$, $\BZ$ and $\BQ$ denote the sets of non-negative integers,
integers and rational numbers respectively.


\section*{Acknowledgement} This paper is partly supported by
the National Natural Science Foundation of China (grant 11671137).

Part of this work was done while the author was a visitor at
Shanghai Center for Mathematical Sciences. The author is grateful to
this institution for its hospitality.

The author thanks Liang Xiao for his helpful discussion.

\section{Fontaine period rings and Galois cohomology}
\label{sec:gal-coh}

Let $\BB_\cris$, $\BB_\st$ and $\BB_\dR$ be Fontaine's period rings
\cite{Fontaine}. We recall their definitions. Let
$\widetilde{\mathrm{E}}^+$ be the ring $\{(x^{(0)}, x^{(1)}, \dots):
x^{(i)}\in \mathcal{O}_{\BC_p}, (x^{(i+1)})^p=x^{(i)} \ \forall \
i\in\BN\}$. Let $\varepsilon$ be an element of
$\widetilde{\mathrm{E}}^+$ such that $\varepsilon^{(0)}=1$ and
$\varepsilon^{(1)}\neq 1$; let $\tilde{p}$ be an element of
$\widetilde{\mathrm{E}}^+$ such that $\tilde{p}^{(0)}=p$. There is a
homomorphism
$$ \theta: \mathrm{W}(\widetilde{\mathrm{E}}^+)\rightarrow \mathcal{O}_{\BC_p}, \hskip 5pt \sum_{n\geq 0}p^n[x_n]\mapsto \sum_{n\geq 0} p^n x_n^{(0)},
$$ where $\mathrm{W}(\widetilde{\mathrm{E}}^+)$ is the ring of Witt
vectors with coefficients in $\widetilde{\mathrm{E}}^+$. Let
$\BB_\dR^+$ be the $\ker(\theta)$-adic completion of
$\mathrm{W}(\widetilde{\mathrm{E}}^+)[1/p]$. Then $
\log[\varepsilon]$ converges in $\BB_\dR^+$, which is denoted by $
t_\cyc $. Put $\BB_\dR=\BB_\dR^+[1/t_\cyc]$. There is a filtration
$\Fil^\bullet$ on $\BB_\dR$ such that $\Fil^i \BB_\dR=\BB_\dR^+
t_\cyc^i$. Let $\mathbf{A}_\cris$ be the $p$-adic completion of the
divided power envelope of $\mathrm{W}(\widetilde{\mathrm{E}}^+)$
with respect to $\ker(\theta)$. Put
$\BB_\cris=\mathbf{A}_\cris[1/t_\cyc]$ and
$\BB_\st=\BB_\cris[\log([\tilde{p}]/p)]$, which are two subrings of
$\BB_\dR$. The Frobenius on $\widetilde{\mathrm{E}}^+$ induces
operators on $\BB_\cris$ and $\BB_\st$ denoted by $\varphi$. Let $N$
be the $\BB_\cris$-derivation on $\BB_\st$, called the {\it
monodromy}, such that $N(\log([\tilde{p}]/p))=-1$. The action of
$G_{\BQ_p}$ on $\mathcal{O}_{\BC_p}$ induces actions on
$\widetilde{\mathrm{E}}^+$, $\BB_\cris$, $\BB_\st$ and $\BB_\dR$;
the action of $G_{\BQ_p}$ on $\BB_\st$ commutes with $\varphi$ and
$N$.

Put
$$\BB_{\cris,E}=\BB_\cris\otimes_{\BQ_p}E, \ \ \BB_{\st,E}=\BB_\st\otimes_{\BQ_p}E,
\ \ \BB_{\dR,E}=\BB_{\dR}\otimes_{\BQ_p}E.$$ We extend the actions
of $G_{\BQ_p}$ on $\BB_{\cris}$, $\BB_\st$ and $\BB_\dR$
$E$-linearly to $\BB_{\cris,E}$, $\BB_{\st,E}$ and $\BB_{\dR,E}$. We
also extend the operators $\varphi$ and $N$ on $\BB_\st$
$E$-linearly to $\BB_{\st,E}$. Then $\BB_{\cris,E}$ is stable under
$\varphi$ and $\BB_{\cris, E}=\BB_{\st,E}^{N=0}$. We have
$\varphi(t_\cyc)=p \, t_\cyc$, $Nt_\cyc=0$ and
$g(t_\cyc)=\chi_\cyc(g)t_\cyc$ for $g\in G_{\BQ_p}$. Let
$\Fil^\bullet$ be the filtration on $\BB_{\dR,E}$ such that $
\Fil^i\BB_{\dR,E} = \Fil^i \BB_\dR \otimes_{\BQ_p} E $. Put
$\BB^+_{\dR,E}=\BB^+_\dR\otimes_{\BQ_p}E=\Fil^0 \BB_{\dR,E}$. Then
we have the following short exact sequence, the so called {\it
fundamental exact sequence} \cite[Proposition 1.3 v)]{CzFon}
$$\xymatrix{ 0 \ar[r] & E \ar[r] & \BB^{\varphi=1}_{\cris,E} \ar[r] & \BB_{\dR,E}/ \BB_{\dR,E}^+ \ar[r] & 0.
}$$

The following lemma is well known. See \cite[Proposition
1.1]{Cz2010}. 

\begin{lem}\label{lem:coh-wellknown}
Let $a\leq b$ be in $\BZ\cup \{-\infty, +\infty\}$. If either $a>0$ or $b\leq
0$, then
$$ H^0(\Fil^a \BB_{\dR,E} /\Fil^b \BB_{\dR,E})= H^1(\Fil^a \BB_{\dR,E}/\Fil^b \BB_{\dR,E})=0
$$ with the convention $\Fil^{-\infty}\BB_{\dR,E}=\BB_{\dR,E}$ and
$\Fil^{+\infty}\BB_{\dR,E}=0$.
\end{lem}

For $i\in \BN$ and $j\in \BZ$ put $U_{i,j}=\BB_{\st,E}^{N^{i+1}=0,
\varphi=p^j }$. Note that $U_{i,i-1}$ coincides with the notation
$U_i$ in \cite{Cz2010}.

\begin{lem} \label{lem:exact-Bcris-BdR} For every $i\geq 1$ we have the following short exact sequence
\[ \xymatrix{ 0 \ar[r] & \BB_{\cris,E}^{\varphi=p^j} \ar[r] & U_{i, j} \ar[r]^{N\ \ } & U_{i-1,j-1}\ar[r] & 0.  }\]
\end{lem}
\begin{proof} We only need to prove the surjectivity of $N: U_{i,j}\rightarrow U_{i-1, j-1}$.
Let $u$ be the element in $\BB_{\st}$, considered as an element in
$\BB_{\st,E}$, that is denoted by $\log[\pi]$ in \cite[\S
1.5]{CzFon}. Then $\BB_{\st,E}=\BB_{\cris,E}[u]$ and $\varphi
(u)=pu$, $N(u)=-1$. For $x\in U_{i-1,j-1}$ write
$x=\sum_{\ell=0}^{i-1} a_\ell u^\ell$ with $a_\ell\in \BB_{\cris,
E}$. Then $a_\ell$ is in $\BB_{\cris,E}^{\varphi=p^{i-1-\ell}}$. So
$y=-\sum_{\ell=0}^{i-1} a_\ell \frac{u^{\ell+1}}{\ell+1} $ is in
$U_{i,j}$ and $N(y)=x$.
\end{proof}

\begin{prop}\label{prop:isom}
If $i\geq 1$, then the inclusion $E\subset U_{i,0}$ induces an
isomorphism
$$ H^1(E) \xrightarrow{\sim} \ker(H^1(U_{i,0})\xrightarrow{N} H^1(\BB_{\st,E}) ). $$
\end{prop}
\begin{proof} We prove the assertion by induction on $i$. For $i=1$,
the assertion is \cite[Proposition 1.2]{Cz2010}.

By definition $U_{0,-i}=\BB_{\cris,E}^{\varphi=p^{-i}}$. From the
fundamental exact sequence we obtain the following exact sequence
\[\xymatrix{ 0 \ar[r] & E t^{-i}\ar[r] & U_{0,-i} \ar[r] & \BB_{\dR,E}/ \Fil^{-i} \BB_{\dR, E} \ar[r] &
0. }\] So we have an isomorphism $H^1(E (-i))=
H^1(Et^{-i})\rightarrow H^1(U_{0,-i})$ since by Lemme
\ref{lem:coh-wellknown}
$$H^0(\BB_{\dR,E}/ \Fil^{-i} \BB_{\dR, E})=H^1(\BB_{\dR,E}/ \Fil^{-i}
\BB_{\dR, E})=0$$ for $i\geq 0$. When $i\geq 1$, each nontrivial
extension of $E$ by $E(-i)$ is not semistable. (This is a well known
fact; it also follows from Proposition \ref{prop:ht>0-a} below.)
Thus $H^1(E(-i))\rightarrow H^1(\BB_{\st,E})$ induced by the natural
inclusion $E(-i)\subset \BB_{\st,E}$ is injective. As $H^1(E
(-i))\rightarrow H^1(U_{0,-i})$ is an isomorphism, it follows that
$H^1(U_{0,-i})\rightarrow H^1(\BB_{\st,E})$ is also injective.

Note that $$\ker(H^1(U_{i,0})\xrightarrow{N}
H^1(\BB_{\st,E}))\subset \ker(H^1(U_{i,0})\xrightarrow{N^i}
H^1(\BB_{\st,E})).$$ We consider the exact sequence
\[ \xymatrix{  0 \ar[r] & U_{i-1,0} \ar[r] & U_{i,0} \ar[r]^{N^i} & U_{0, -i} \ar[r] & 0.
}\] As $H^0(U_{0,-i})=0$, from this short exact sequence we derive
an isomorphism
$$ H^1(U_{i-1,0})\xrightarrow{\sim} \ker (H^1(U_{i,0})\xrightarrow{N^i} H^1(U_{0,-i})).
$$ In particular the natural map $H^1(U_{i-1,0})\rightarrow
H^1(U_{i,0})$ is injective. As $H^1(U_{0,-i})$ injects into
$H^1(\BB_{\st,E})$, we have
$$ \ker(H^1(U_{i,0})\xrightarrow{N^i}
H^1(\BB_{\st,E}))=\ker(H^1(U_{i,0})\xrightarrow{N^i} H^1(U_{0,-i})).
$$ It follows that $\ker(H^1(U_{i,0})\xrightarrow{N}
H^1(\BB_{\st,E}))$ lies in the image of $H^1(U_{i-1,0})\rightarrow
H^1(U_{i,0})$. Since $H^1(U_{i-1,0})$ injects into $H^1(U_{i,0})$,
we have an isomorphism
$$ \ker(H^1(U_{i-1,0})\xrightarrow{N} H^1(\BB_{\st,E}))
\xrightarrow{\sim} \ker(H^1(U_{i,0})\xrightarrow{N}
H^1(\BB_{\st,E})).
$$ This completes the inductive proof.
\end{proof}

\begin{cor} The inclusion $E \subset \BB_{\st,E}^{\varphi=1}$
induces an isomorphism
$$ H^1(E) \xrightarrow{\sim} \ker(H^1(\BB_{\st,E}^{\varphi=1})\xrightarrow{N} H^1(\BB_{\st,E})) .
$$
\end{cor}
\begin{proof} First we prove that $H^1(E)\rightarrow
H^1(\BB_{\st,E}^{\varphi=1})$ is injective. Let $c$ be a $1$-cocycle
with values in $E$. If the image of $[c]$ in
$H^1(\BB_{\st,E}^{\varphi=1})$ is zero, then there exists some $y\in
\BB_{\st,E}^{\varphi=1}$ such that $c_\sigma=(\sigma-1)y$ for all
$\sigma\in G_{\BQ_p}$. As $\BB_{\st,E}^{\varphi=1}=\cup_{i\geq
1}U_{i,0}$, $y$ is in $U_{i,0}$ for some $i\geq 1$, which implies
that the image of $[c]$ in $H^1(U_{i,0})$ is zero. But by
Proposition \ref{prop:isom}, $H^1(E)$ injects to $H^1(U_{i,0})$, so
$[c]=0$ (in $H^1(E)$).

Now, let $c$ be a $1$-cocycle with values in
$\BB_{\st,E}^{\varphi=1}$ such that the image of $[c]$ by $ N:
H^1(\BB_{\st,E}^{\varphi=1}) \rightarrow H^1(\BB_{\st,E})$ is zero.
Then there exists some $z\in \BB_{\st,E}$ such that $N(c_\sigma)=
(\sigma-1)z$ for all $\sigma\in G_{\BQ_p}$. Let $i$ be a positive
integer such that $N^i(z)=0$. Then $c_\sigma\in U_{i,0}$ for all
$\sigma\in G_{\BQ_p}$. In other words, $[c]$ comes from an element
in $\ker(H^1(U_{i,0})\xrightarrow{N} H^1(\BB_{\st,E}))$ by the map
$H^1(U_{i,0})\rightarrow H^1(\BB_{\st,E}^{\varphi=1})$. So by
Proposition \ref{prop:isom}, $[c]$ comes from an element in $H^1(E)$
by the map $H^1(E)\rightarrow H^1(\BB_{\st,E}^{\varphi=1})$.
\end{proof}

\section{Some facts on Galois representations}\label{sec:gal-rep}

Throughout this section a {\it filtration} on an $E$-vector space
$D$ means an exhaustive descending $\BZ$-indexed filtration.

\subsection{$\BX_\st$ and $\BX_\dR$}

We will use the functors $\BX_{\st}$ and $\BX_{\dR}$ defined in
\cite{Cz2010}. These functors were already used in \cite{CzFon} to
show that any admissible filtered $(\varphi,N)$-module comes from a
Galois representation. In \cite{CzFon} $\BX_\st$ and $\BX_\dR$ are
denoted by $V^0_{st}$ and $V^1_{st}$ respectively.

First we recall the notions of $E$-$(\varphi,N)$-modules, filtered
$E$-$(\varphi,N)$-modules and admissible filtered
$E$-$(\varphi,N)$-modules \cite{BM} (see also \cite{Cz2010}). They
are the $E$-linear versions of $(\varphi,N)$-modules, filtered
$(\varphi,N)$-modules and admissible filtered $(\varphi,N)$-modules
defined by Fontaine \cite{Fontaine-b}.

An {\it $E$-$(\varphi,N)$-module} is an $E$-vector space with
$E$-linear operators $\varphi$ and $N$ such that $N \varphi=p\varphi
N$. If $D$ is a finite-dimensional $E$-$(\varphi,N)$-module, we put
$$ t_N(D)=v_p(\det \varphi)$$ where $v_p$ is the valuation on $E$
such that $v_p(p)=1$.

A {\it filtration} $\Fil^\bullet$ on an $E$-vector space $D$ is a
collection of $E$-vector subspaces $\Fil^jD$ ($j\in\BZ$) such that
$\Fil^{j+1}D\subset \Fil^j D$, $\cap_{j\in\BZ}\Fil^jD=0$ and
$\cup_{j\in\BZ}\Fil^jD=D$. A {\it filtered $E$-module} is an
$E$-vector space with a filtration. If $(D,\Fil) $ is a
finite-dimensional filtered $E$-module, we put
$$ t_H(D,\Fil)=\sum_{j\in\BZ}j \dim_E(\Fil^jD/\Fil^{j+1}D).$$

A {\it filtered $E$-$(\varphi,N)$-module} $(D,\Fil)$ is an
$E$-$(\varphi,N)$-module $D$ equipped with a filtration $\Fil$. Such
a module is called {\it admissible} if $D$ is finite-dimensional
over $E$, if $\varphi$ is invertible and if $t_H(D,\Fil)=t_N(D)$ and
$t_H(D',\Fil)\leq t_N(D')$ for every $E$-$(\varphi,N)$-submodule
$D'$ equipped with the induced filtration.

Note that, if $D$, $D_1$ and $D_2$ are filtered
$E$-$(\varphi,N)$-modules, then there exist natural filtered
$E$-$(\varphi,N)$-module structures on $D^*$ and on $D_1\otimes_E
D_2$.

If $V$ is a finite-dimensional $E$-representation of $G_{\BQ_p}$,
then $\BD_\st(V)=(\BB_{\st,E}\otimes_E V)^{G_{\BQ_p}}$ is a filtered
$E$-$(\varphi,N)$-module induced from the natural filtered
$E$-$(\varphi,N)$-module structure on $\BB_{\st,E}\otimes_E V$. We
always have $\dim_E \BD_\st(V)\leq \dim_E V$, and say that $V$ is
{\it semistable} if $\dim_E\BD_\st(V)=\dim_E V$.

If $D$ is an $E$-$(\varphi,N)$-module, let $\BX_\st(D)$ be the
$\BB_{\cris,E}^{\varphi=1}$-module defined by
$$\BX_\st(D)=(\BB_{\st,E}\otimes_E D)^{\varphi=1,N=0} . $$
If $(D,\Fil)$ is a filtered $E$-module, let $\BX_\dR(D, \Fil)$ or
just $\BX_\dR(D)$ if there is no confusion, be the $\BB^+_{\dR,
E}$-module
$$\BX_\dR(D,\Fil) = (\BB_{\dR,E}\otimes_E D)/\Fil^0(\BB_{\dR,E}\otimes_E D) .  $$
By \cite[Proposition 5.1, Proposition 5.2]{CzFon} $\BX_\st$ and
$\BX_\dR$ are exact.

If $(D,\Fil)$ is a filtered $E$-$(\varphi,N)$-module, then there is
a natural $E$-linear map $\BX_\st(D)\rightarrow \BX_\dR(D,\Fil)$
induced by the inclusion $\BB_{\st,E}\otimes_E D\rightarrow
\BB_{\dR, E}\otimes_E D$. Let $\BV_\st(D,\Fil)$ be the kernel of
this map, which is an $E$-vector space.

By \cite[Theorem 2.1]{Cz2010} $\BV_{\st}$ is an equivalence of
categories from the category of admissible filtered
$E$-$(\varphi,N)$-modules to the category of semistable
$E$-representations of $G_{\BQ_p}$, with quasi-inverse $\BD_\st$.
Furthermore, $\BV_\st$ and $\BD_\st$ respect tensor products and
duals.

If $(D,\Fil)$ is an admissible filtered $E$-$(\varphi,N)$-module,
then the sequence \begin{equation}\label{eq:XstXdR-exact} \xymatrix{
0\ar[r] & \BV_\st(D,\Fil) \ar[r] & \BX_\st(D) \ar[r] &
\BX_\dR(D,\Fil) \ar[r] & 0 }\end{equation} is exact, and the natural
map
$$ \BB_{\st,E}\otimes_E \BV_\st(D,\Fil) \rightarrow \BB_{\st,E}\otimes_E D
$$ is an isomorphism respecting the actions of $G_{\BQ_p}$, $\varphi$, $N$ and the
filtrations.

If $D$ is an $E$-$(\varphi,N)$-module, and $e^*$ is an element in
the dual $E$-$(\varphi,N)$-module $D^*$, we have a
$G_{\BQ_p}$-equivariant map $$  \pi_{e^*}: \BX_{\st}(D)\rightarrow
\BB_{\st,E}, \ \ \ x\mapsto <e^*, x>.
$$ Here $<\cdot, \cdot>$ denotes the $\BB_{\st,E}$-bilinear pairing
$$(\BB_{\st,E}\otimes_E D^*)\times (\BB_{\st,E}\otimes_E
D)\rightarrow \BB_{\st,E}$$ induced by the canonical $E$-bilinear
pairing $D^*\times D\rightarrow E$.

\begin{lem} \label{lem:Bst-proj}
\begin{enumerate}
\item\label{it:Bst-proj-a} We have $N \circ \pi_{e^*}=\pi_{N e^*}$.
\item\label{it:Bst-proj-b} If $N^{i+1}e^*=0$, then the image of $\pi_{e^*}$ is in
$\BB_{\st,E}^{N^{i+1}=0}$. If furthermore $\varphi(e^*)=e^*$, then
the image of $\pi_{e^*}$ is in $U_{i,0}$.
\end{enumerate}
\end{lem}
\begin{proof} For every $x\in \BX_{\st}(D)$, as $Nx=0$, we have
$N<e^*,x>=<Ne^*,x>$.

If $N^{i+1}e^*=0$, then $N^{i+1}<e^*,x>=<N^{i+1}e^*,x>=0$. If
$\varphi e^*=e^*$, then $\varphi <e^*,x>=<\varphi e^*, \varphi
x>=<e^*,x>$ since $\varphi(x)=x$. So $<e^*,x>$ is in $U_{i,0}$.
\end{proof}

For $e^*\in D^*$, $\pi_{e^*}$ induces a map $
H^1(\BX_\st(D))\rightarrow H^1(\BB_{\st,E})$ again denoted by
$\pi_{e^*}$. If $\varphi(e^*)=e^*$,  $\pi_{e^*}$ induces  a map
$H^1(\BX_\st(D))\rightarrow H^1(\BB_{\st,E}^{\varphi=1})$ which will
be denoted by $\tilde{\pi}_{e*}$.

\subsection{Exactness of $H^i(\BX_\dR(-))$}

Let $D$ be a finite-dimensional $E$-vector space. For a basis
$\{e_1,\dots, e_n\}$ of $D$ over $E$ and a filtration $\Fil$ on $D$
we say that $\{e_1,\dots, e_n\}$ is {\it compatible} with $\Fil$ if
for every $i$, $\Fil^iD= \oplus_{j=1}^n \Fil^iD\cap Ee_j$. If we
write $f_j$ ($j=1,\dots, n$) for the largest integer such that
$\Fil^{f_j}D \cap E e_j\neq 0$, then $\{e_1,\dots, e_n\}$ is
compatible with $\Fil$, if and only if $\Fil^i
D=\bigoplus\limits_{j:f_j\geq i} E e_j$ for every $i$. In this case
we have
$$ \Fil^i (\BB_{\dR,E}\otimes_E D) = \bigoplus_j \Fil^{i-f_j}B_{\dR,E} \cdot e_j. $$

\begin{lem}\label{lem:basis}
Let $$\xymatrix{ 0 \ar[r] & D_1 \ar[r] & D \ar[r] & D_2\ar[r] & 0
}$$ be a short exact sequence of filtered $E$-modules. If
$\{e_1,\dots, e_s\}$ is a basis of $D_1$ and $\{\bar{e}_{s+1},\dots,
\bar{e}_n\}$ is a basis of $D_2$ over $E$ compatible with the
filtration on $D_1$ and that on $D_2$ respectively, then there exist
liftings $e_j$ of $\bar{e}_j$ $(j=s+1,\dots, n)$ such that
$\{e_1,\dots, e_n\}$ is a basis of $D$ compatible with the
filtration.
\end{lem}
\begin{proof} Let $f_j$ ($j=1,\dots, n$) be the largest integer
such that $\Fil^{f_j}D_1 \cap E e_j \neq 0$ for $j=1,\dots, s$, and
$\Fil^{f_j}D_2 \cap E \bar{e}_i\neq 0$ for $j=s+1,\dots, n$. As the
filtration on $D_2$ is induced from that on $D$, there exists a
lifting $e_j$ of $\bar{e}_j$ in $\Fil^{f_j}D$ for every
$j=s+1,\dots, n$. Then $\bigoplus\limits_{j: f_j \geq i} E e_j$ is
contained in $\Fil^i D$. However, we have \begin{eqnarray*} \dim_E
\Fil^i D &=& \dim_E \Fil^i D_1 +\dim_E \Fil^i D_2 \\ & = & \sharp\{
j: 1\leq j\leq s, \ f_j\geq i \} + \sharp\{ j: s+1\leq j\leq n, \
f_j\geq i \} \\ &=& \sharp\{ j: 1\leq j\leq n, \ f_j\geq i \}
.\end{eqnarray*} Therefore $\Fil^iD = \bigoplus\limits_{j: f_j \geq
i} E e_j $.
\end{proof}

\begin{prop}\label{prop:split} If $$\xymatrix{0\ar[r] & D_1 \ar[r] & D\ar[r] & D_2 \ar[r] &
0}$$ is a short exact sequence of filtered $E$-modules, then
$$ \xymatrix{0\ar[r] & \BX_\dR(D_1) \ar[r] & \BX_\dR(D) \ar[r] & \BX_\dR(D_2) \ar[r] &
0} $$ is a split short exact sequence of $G_{\BQ_p}$-modules.
\end{prop}
\begin{proof} Let $\{e_j\}_{j=1}^n$, $\{\bar{e}_j\}_{j=s+1}^n$ and $\{f_j\}_{j=1}^n$ be
as in Lemma \ref{lem:basis} and its proof. By the definition of
$\BX_\dR(-)$ we have
\begin{eqnarray*} \BX_\dR(D_1) & = & \oplus_{j=1}^s
(\BB_{\dR,E}/\Fil^{f_j}\BB_{\dR,E})\cdot e_j , \\ \BX_\dR(D) & = &
\oplus_{j=1}^n (\BB_{\dR,E}/\Fil^{f_j}\BB_{\dR,E})\cdot e_j , \\
\BX_\dR(D_2) & = & \oplus_{j=s+1}^n
(\BB_{\dR,E}/\Fil^{f_j}\BB_{\dR,E})\cdot \bar{e}_j .
\end{eqnarray*} As $\{e_j\}_{j=1}^n$ and $\{\bar{e}_j\}_{j=s+1}^n$
are fixed by $G_{\BQ_p}$, our assertion is clear now.
\end{proof}

The following follows directly from Proposition \ref{prop:split}.

\begin{cor} \label{cor:dR-exact} If $$\xymatrix{0\ar[r] & D_1 \ar[r] & D\ar[r] & D_2 \ar[r] &
0}$$ is a short exact sequence of filtered $E$-modules, then
$$ \xymatrix{0\ar[r] & H^i(\BX_\dR(D_1)) \ar[r] & H^i(\BX_\dR(D)) \ar[r] & H^i(\BX_\dR(D_2)) \ar[r] &
0} $$ is exact.
\end{cor}

\subsection{The kernel of $H^1(\BX_\st(D))\rightarrow H^1(\BX_\dR(D,\Fil))$}

In this subsection we study the kernel of
$H^1(\BX_\st(D))\rightarrow H^1(\BX_\dR(D,\Fil))$. When $(D,\Fil)$
is admissible, from the short exact sequence (\ref{eq:XstXdR-exact})
we see that this kernel coincides with the image of
$H^1(\BV_\st(D))\rightarrow H^1(\BX_\st(D))$.

We fix a finite-dimensional $E$-$(\varphi,N)$-module $D$. For two
filtrations $\Fil_1$ and $\Fil_2$ on $D$, we write $\Fil_1\approx
\Fil_2$ if $\Fil^0_1D=\Fil^0_2D$. Then $\approx$ is an equivalence
relation on the set of filtrations on $D$.

\begin{prop}\label{prop:same-kernel}
If $\Fil_1\approx \Fil_2$, then the kernel of $H^1(\BX_\st(D))\rightarrow H^1(\BX_\dR(D,
\Fil_1))$ coincides with the kernel of $H^1(\BX_\st(D))\rightarrow
H^1(\BX_\dR(D,\Fil_2))$.
\end{prop}
\begin{proof} By \cite[Proposition 3.1]{CzFon} there exists a basis $\{e_1,\dots,
e_n\}$ of $D$ compatible with both $\Fil_1$ and $\Fil_2$. Write
$f_{j,\ell}$ ($j=1,\dots, n$ and $ \ell=1,2$) for the largest
integer such that
$$\Fil^{f_{j,\ell}}_\ell D \cap E e_j\neq 0.$$  Put $\bar{f}_j=\min(f_{j,1}, f_{j,2})$.

Put $$ M = (\BB_{\dR,E}\otimes_E D )/ \Fil^0_1 (\BB_{\dR,E}\otimes_E
D )\cap \Fil^0_2(\BB_{\dR,E}\otimes_E D). $$ Note that
\begin{eqnarray*}
&& \Fil^0_\ell (\BB_{\dR,E}\otimes_E D ) = \bigoplus_{j=1}^n
\Fil^{-f_{j,\ell}}\BB_{\dR,E} \cdot e_j  , \ \  \ell=1,2, \\  &&
\Fil^0_1 (\BB_{\dR,E}\otimes_E D )\cap \Fil^0_2(\BB_{\dR,E}\otimes_E
D) =\bigoplus_{j=1}^n \Fil^{-\bar{f}_j}\BB_{\dR,E} \cdot e_j .
\end{eqnarray*}
So, for $\ell=1,2$ we have an exact sequence
\[\xymatrix{ 0 \ar[r] & \bigoplus_{j=1}^n (\Fil^{-f_{j,\ell}}\BB_{\dR,E}/\Fil^{-\bar{f}_{j}}\BB_{\dR,E})\cdot e_j
\ar[r] & M \ar[r] &  X_\dR(D, \Fil_\ell) \ar[r] & 0 .}\] As
$\Fil^0_1D=\Fil^0_2D$, $f_{j,1}\geq 0$ if and only if $f_{j,2}\geq
0$. Thus $\bar{f}_j\geq 0$ if and only if $f_{j,\ell}\geq 0$. So by
Lemma \ref{lem:coh-wellknown} we have
$$ H^1(\Fil^{-f_{j,\ell}}\BB_{\dR,E}/ \Fil^{-\bar{f}_{j}}\BB_{\dR,E}) =0.  $$
As a consequence, $ H^1(M)\rightarrow H^1(\BX_\dR(D,\Fil_\ell))$ is
injective.

Note that $\BX_\st(D)\rightarrow \BX_\dR(D, \Fil_\ell)$ $(\ell=1,2)$
factors through $\BX_\st(D)\rightarrow M$. Thus
$H^1(\BX_\st(D))\rightarrow H^1(\BX_\dR(D,\Fil_\ell))$ factors
through $H^1(\BX_\st(D))\rightarrow H^1(M)$. Since $H^1(M)$ injects
to $H^1(\BX_\dR(D,\Fil_\ell))$, the kernel of
$H^1(\BX_\st(D))\rightarrow H^1(\BX_\dR(D,\Fil_\ell))$ coincides
with the kernel of $H^1(\BX_\st(D))\rightarrow H^1(M)$.
\end{proof}

\subsection{The map $H^1(V)\rightarrow H^1(\BB_{\st,E}\otimes_E V)$}

\begin{prop}\label{prop:ht>0-a}
 If $V$ is a semistable $E$-representation of $G_{\BQ_p}$ with
Hodge--Tate weights $>0$, then any nontrivial extension of the
trivial representation $E$ of $G_{\BQ_p}$ by $V$ is not semistable.
\end{prop}
\begin{proof} The
filtered $E$-$(\varphi,N)$-module attached to the trivial
representation $E$ is $D_0=E\cdot e_0$ with
\begin{eqnarray*} \varphi e_0=e_0, \ \ Ne_0=0, \ \  \Fil^0 D_0=D_0, \ \ \Fil^1 D_0=0.
\end{eqnarray*} Let $\widetilde{V}$ be an extension of $E$ by $V$ that is
a semistable representation of $G_{\BQ_p}$. Let $D$ and
$\widetilde{D}$ be the filtered $E$-$(\varphi,N)$-module attached to
$V$ and that attached to $\widetilde{V}$ respectively. Then we have
an exact sequence of filtered $E$-$(\varphi,N)$-modules
\begin{equation} \label{eq:fil-var-mod}
\xymatrix{ 0 \ar[r] & D \ar[r] & \widetilde{D} \ar[r] & D_0\ar[r] &
0.} \end{equation} Let $\{e_1,\dots, e_n\}$ be a basis of $D$ over
$E$, and let $A=(a_{ij})$ be the matrix of $\varphi$ with respect to
this basis so that $\varphi(e_i)=\sum_{j=1}^n a_{ji}e_j$. As $V$ is
of Hodge--Tate weights $>0$, we have $\Fil^1D=D$. By the fact that
the Newton polygon is above the Hodge polygon, the lowest slope of
eigenvalues of $A$ is positive. Thus $I_n-A$ is invertible, where
$I_n$ is the unit $n\times n$-matrix.

Let $\widetilde{e}$ be any lifting of $e_0$. Since
$\varphi(e_0)=e_0$, there are $c_1,\dots, c_n\in E$ such that
$\varphi(\widetilde{e})=\widetilde{e} +\sum\limits_{i=1}^n c_i e_i$.
As $I_n-A$ is invertible, there is a unique vector $(b_1,\dots,
b_n)^t$ such that $(I_n-A) \cdot (b_1,\dots, b_n)^t=(c_1,\dots,
c_n)^t$. Then $e=\widetilde{e}+\sum\limits_{i=1}^n b_i e_i$
satisfies $\varphi(e)=e$. From the relation $N\varphi=p\varphi N$ we
obtain $Ne\in D^{\varphi=p^{-1}}=0$. As $\Fil ^1 D_0=0$, we have
$e\notin \Fil^1 \widetilde{D}$. Hence the exact sequence
(\ref{eq:fil-var-mod}) splits and so $\widetilde{V}$ is a trivial
extension of $E$ by $V$.
\end{proof}

\begin{cor} \label{cor:ht>0} Let $V$ be a semistable $E$-representation of $G_{\BQ_p}$ with Hodge--Tate
weights $>0$. Then the following hold:
\begin{enumerate}
\item \label{it:coh-st-a}
The natural map $H^1(V)\rightarrow H^1(\BB_{\st,E}\otimes_E V)$ is
injective.
\item \label{it:coh-st-b}
Let $c$ be in $H^1(V)$. If for every $f\in
\mathrm{Hom}_{G_{\BQ_p}}(V, \BB_{\st,E})$, the image of $c$ by the
map $H^1(V)\rightarrow H^1(\BB_{\st,E})$ induced by $f$ is zero,
then $c=0$.
\end{enumerate}
\end{cor}
\begin{proof} Assertion (\ref{it:coh-st-a}) follows immediately from
Proposition \ref{prop:ht>0-a}.

Next we prove (\ref{it:coh-st-b}). Let $D$ be the filtered
$E$-$(\varphi,N)$-module attached to $V$, and let $\{e_1,\dots,
e_n\}$ be a basis of $D$ over $E$. Let $\pi_i$ denote the projection
$$\BB_{\st,E}\otimes_E D=\BB_{\st,E}\otimes_E V \rightarrow
\BB_{\st,E} , \ \ \ \sum_{j=1}^n a_j e_j \mapsto a_i.$$ As
$e_1,\dots, e_n$ are fixed by $G_{\BQ_p}$, $\pi_i$ $(i=1,\dots, n)$
are in $\mathrm{Hom}_{G_{\BQ_p}}(\BB_{\st,E}\otimes_E V,
\BB_{\st,E})$. So the composition of $\pi_i$ and the inclusion
$V\hookrightarrow \BB_{\st,E}\otimes_E V$, denoted by
$\tilde{\pi}_i$, is in $\mathrm{Hom}_{G_{\BQ_p}}(V,\BB_{\st,E})$. In
fact, $\{\tilde{\pi}_1,\dots, \tilde{\pi}_n\}$ is a basis of
$\mathrm{Hom}_{G_{\BQ_p}}(V,\BB_{\st,E})$. Now the condition
$\tilde{\pi}_i(c)=0$ for $i=1,\dots, n$ ensures that the image of
$c$ in $H^1(\BB_{\st,E}\otimes_E V)$ is zero. By (\ref{it:coh-st-a})
we obtain $c=0$.
\end{proof}

\begin{rem}\label{rem:semi-st-ht-pos} If $V$ is semistable with
Hodge--Tate weights $\geq 0$, then the natural map
$$ H^1(V) \rightarrow H^1(\BX_\st(\BD_\st(V))) $$ is an isomorphism.
\end{rem}
\begin{proof} By (\ref{eq:XstXdR-exact}) we have a short exact sequence
\[ \xymatrix{ 0 \ar[r] & V \ar[r] & \BX_\st(\BD_\st(V)) \ar[r] & \BX_\dR(\BD_\st(V)) \ar[r] & 0
,}\] from which we obtain an exact sequence
\[\xymatrix{ H^0(\BX_\dR(\BD_\st(V))) \ar[r] & H^1(V) \ar[r] & H^1(\BX_\st(\BD_\st(V))) \ar[r] &
H^1(\BX_\dR(\BD_\st(V))). }\] As $V$ is of Hodge--Tate weight $\geq
0$, we have $H^0(\BX_\dR(\BD_\st(V)))=H^1(\BX_\dR(\BD_\st(V)))=0$.
\end{proof}

\section{Triangulations and refinements}\label{sec:tri-ref}

We recall the theory of triangulations and refinements \cite{BC,
Ben, Be2011, Cz2008}.

To explain what a triangulation is, we need the theory of
$(\varphi,\Gamma)$-modules. In $p$-adic Hodge theory one considers
$(\varphi,\Gamma)$-modules as the category of semilinear algebra
data describing $p$-adic Galois representations. The
$(\varphi,\Gamma)$-modules are modules over various rings of power
series (denoted by $\mathscr{E}$, $\mathscr{E}^\dagger$ and $\CR$).
See \cite{Fontaine1990, CC, Ked05} for precise constructions of
these rings and definitions of $(\varphi,\Gamma)$-modules.

\begin{thm} $($\cite{Fontaine1990, CC, Ked05}$)$ There is an equivalence of categories between the
category of $p$-adic representations of $G_{\BQ_p}$ and the category
of \'etale $(\varphi,\Gamma)$-modules over either $\mathscr{E}$,
$\mathscr{E}^\dagger$ or $\mathscr{R}$.
\end{thm}

In \cite{BeCo} Berger and Colmez defined a functor from the category
of families of $p$-adic Galois representations to the category of
families of overconvergent \'etale $(\varphi,\Gamma)$-modules. After
it, in \cite{KL} Kedlaya and Liu defined a functor $\BD_\rig$ from
the category of families of $p$-adic Galois representations to the
category of families of \'etale $(\varphi,\Gamma)$-modules over the
Robba ring $\CR$. But neither of these two functors is an
equivalence of categories, in contrast with the classical case.

For the convenience of the reader we recall the definitions of
families of Galois representations and families of \'etale
$(\varphi,\Gamma)$-modules.

If $S$ is an affinoid $E$-algebra, by an {\it $S$-representation} of
$G_{\BQ_p}$ we mean a locally free $S$-module of finite constant
rank equipped with a continuous $S$-linear action of $G_{\BQ_p}$.
Let $\CR_S$ be the Robba ring over $S$ which is a topological ring
equipped with continuous actions of $\varphi$ and $\Gamma$
\cite{KL}. By a {\it $($locally$)$ free $(\varphi,\Gamma)$-module}
over $\CR_S$ we mean a (locally) free $\CR_S$-module $\CM$ of finite
constant rank equipped with a semilinear action of $\varphi$ such
that the map $\varphi^*\CM\rightarrow \CM$ is an isomorphism, and a
semilinear action of $\Gamma$ that commutes with the
$\varphi$-action and is continuous for the profinite topology on
$\Gamma$ and the topology on $\CR_S$. See \cite[Definition 6.3]{KL}
for the notion of \'etale $(\varphi,\Gamma)$-modules over $\CR_S$.
We always consider an $S$-representation of $G_{\BQ_p}$ as a family
of $E$-representations of $G_{\BQ_p}$ over $\Max(S)$, and consider
an \'etale $(\varphi,\Gamma)$-module over $\CR_S$ as a family of
\'etale $(\varphi,\Gamma)$-modules over $\Max(S)$.

We recall the construction of $(\varphi,\Gamma)$-modules over
$\CR_S$ of rank $1$, which play the important role in the definition
of triangulations below.

If $\delta$ is a continuous $S^\times$-valued character of
$\BQ_p^\times$, we let $\CR_S(\delta)$ denote the rank one
$(\varphi,\Gamma)$-module over $\CR_S$, defined by
$\CR_S(\delta)=\CR_S e$ with $\gamma(e)=\delta(\chi_\cyc(\gamma))e$
and $\varphi(e)=\delta(p)e$. By \cite[Appendix]{KPX} every
$(\varphi,\Gamma)$-module over $\CR_S$ of rank $1$ is of this form.
Let $\log (\delta|_{\BZ_p^\times})$ be the logarithm of
$\delta|_{\BZ_p^\times}$, which is an additive character of
$\BZ_p^\times$ with values in $S$. There exists $w_\delta\in S$ such
that $\log (\delta|_{\BZ_p^\times}) = w_\delta
\psi_2|_{\BZ_p^\times}$. We call $w_\delta$ the {\it weight
$($function$)$} of $\delta$. For every $z\in \Max(S)$, if
$\CR_S(\delta_z)$ corresponds to a semistable $E_z$-representation
$V_z$ of $G_{\BQ_p}$, then the Hodge--Tate weight of $V_z$ is
$-w_\delta(z)$.

\begin{defn} \label{def:tri} $($\cite{KPX}$)$
Let $\CM$ be a free $(\varphi,\Gamma)$-module over $\CR_S$ of rank
$n$. If there are
\begin{quote}
$\bullet$ a strictly increasing filtration $$ \{0\}= \Fil_0 D
\subset \Fil_1 D \subset\cdots \subset \Fil_n D =D $$ of saturated
free $\CR_S$-submodule stable by $\varphi$ and $\Gamma$, and

$\bullet$ $n$ continuous characters $\delta_i:
\BQ_p^\times\rightarrow S^\times$ 

\end{quote}
\noindent such that for every $i=1,\dots, n$,
$$\Fil_i\CM/\Fil_{i-1}\CM\simeq \CR_S(\delta_i),$$ 
we say that $\CM$ is {\it triangulable};  we call $\Fil$ a {\it
triangulation} of $\CM$ and
$$(\delta_1,\dots, \delta_n)$$
the {\it triangulation parameters} attached to $\Fil$.
\end{defn}

To discuss the relation between triangulations and refinements, we
restrict ourselves to the case of $S=E$.

Let $\CD$ be a filtered $E$-$(\varphi,N)$-module of rank $n$, and we
assume that all the eigenvalues of $\varphi:\CD\rightarrow \CD$ are
in $E$. Following Mazur \cite{Mazur} we define a {\it refinement} of
$\CD$ to be a filtration on $\CD$
$$ 0 = \CF_0 \CD \subset \CF_1 \CD \subset \cdots \subset \CF_n\CD=\CD $$ by
$E$-subspaces stable by $\varphi$ and $N$, such that each factor
$\mathrm{gr}^\CF_i \CD= \CF_i\CD/\CF_{i-1}\CD$ ($i=1,\dots, n$) is
of dimension $1$. Every refinement fixes an ordering
$\alpha_1,\dots, \alpha_n$ of eigenvalues of $\varphi$ and an
ordering $k_1,\dots, k_n$ of Hodge--Tate weights of $\CD$ taken with
multiplicities such that the eigenvalue of $\varphi$ on
$\mathrm{gr}^\CF_i \CD$ is $\alpha_i$ and the Hodge--Tate weight of
$\mathrm{gr}^\CF_i\CD$ is $k_i$.

\begin{prop} \label{prop:triang-refine} $($\cite[Proposition 1.3.2]{Ben}$)$
Let $\CM$ be a $(\varphi,\Gamma)$-module over $\CR_E$ coming from a
filtered $E$-$(\varphi,N)$-module $\CD$ of dimension $n$ via
Berger's functor \cite{Berger-com}.
\begin{enumerate}
\item The equivalence of categories between the category of semistable
$(\varphi,\Gamma)$-modules and the category of filtered
$E$-$(\varphi,N)$-modules induces a bijection between the set of
triangulations on $\CM$ and the set of refinements on $\CD$.
\item If $(\Fil_i\CM)$ is a triangulation of $\CM$ corresponding to a
refinement $(\CF_i\CD)$ of $\CD$ with the ordering of the
eigenvalues of $\varphi$ being $\alpha_1, \dots, \alpha_n$ and the
ordering of Hodge--Tate weights being $k_1, \dots, k_n$, then for
each $i=1,\dots, n$, $\Fil_i\CM/\Fil_{i-1}\CM$ is isomorphic to
$\CR_E(\delta_i)$ where $\delta_i$ is defined by
$\delta_i(p)=\alpha_ip^{-k_i}$ and $\delta_i(u)=u^{-k_i}$ $(u\in
\BZ_p^\times)$.
\end{enumerate}
\end{prop}

\begin{rem} In \cite{Xie2012} the author gave a family version of
Berger's functor from the category of filtered $(\varphi,N)$-modules
to the category of $(\varphi,\Gamma)$-modules \cite{Berger-com}.
Using this functor we may obtain a family version of Proposition
\ref{prop:triang-refine}. We omit the details since we will not use
it.
\end{rem}

\section{Marked indices and  $\CL$-invariants} \label{sec:L-inv}

Let $D$ be a filtered $E$-$(\varphi,N)$-module of rank $n$. Suppose
that $\varphi$ is semisimple on $D$. Fix a refinement $\CF$ of $D$.
Then $\CF$ fixes an ordering $\alpha_1, \dots, \alpha_n$ of the
eigenvalues of $\varphi$ and an ordering $k_1, \dots, k_n$ of the
Hodge--Tate weights.

\subsection{The operator $N_{\CF}$ and marked indices}

We define an $E$-linear operator $N_\CF$ on $\gr^{\CF}_\bullet
D=\bigoplus\limits_{i=1}^n \CF_{i}D/\CF_{i-1}D$. For every $i\in
\{1, \dots, n\}$, if $N(\CF_i D)=N (\CF_{i-1}D)$, we demand that
$N_\CF$ maps $\gr^\CF_i D$ to zero.

Now we assume that $N(\CF_i D) \supsetneq N (\CF_{i-1} D)$. Let $j$
be the minimal integer such that $$ N (\CF_i D) \subseteq
N(\CF_{i-1}D) + \CF_j D.
$$

\begin{lem}\label{lem:Ncf} We have $N (\CF_{i-1}D) \cap \CF_j D = N (\CF_{i-1}D)\cap
\CF_{j-1}D$.
\end{lem}
\begin{proof} If $N (\CF_{i-1}D) \cap \CF_j D \supsetneq N (\CF_{i-1}D)\cap
\CF_{j-1}D$, then there exists $x\in \CF_{i-1}D$ such that $N (x)$
is in $\CF_jD$ but not in $\CF_{j-1}D$. Then $\CF_{j}D = E \cdot
N(x) \oplus \CF_{j-1}D$. Thus for every $y\in \CF_i D$ we have
$$N(y)\subseteq N (\CF_{i-1}D) + \CF_j D \subseteq N(\CF_{i-1}D) + E
\cdot N(x) + \CF_{j-1}D \subseteq N(\CF_{i-1}D) +
 \CF_{j-1}D.$$ So $N (\CF_i D) \subseteq N(\CF_{i-1}D) + \CF_{j-1}
 D$ which contradicts the minimality of $j$.
\end{proof}

For every $x\in \CF_i D$, if we write $N(x)$ in the form $N(x)= a +
z$ with $a\in N (\CF_{i-1}D)$ and $z\in \CF_j D$, then $z\
\mathrm{mod} \ \CF_{j-1} D$ is uniquely determined. Indeed, if
$N(x)=a'+z'$ is another expression with $a'\in N(\CF_{i-1}D)$ and
$z'\in \CF_j D$, then by Lemma \ref{lem:Ncf} we have
$$ z-z' = a'-a \in N(\CF_{i-1}D) \cap \CF_jD = N(\CF_{i-1}D) \cap
\CF_{j-1}D \subseteq \CF_{j-1}D . $$ We define $$ N_\CF ( x +
\CF_{i-1}D) = z + \CF_{j-1}D \in \gr^{\CF}_j D. $$ It is easy to
check that $$ N_{\CF}(\lambda(x+ \CF_{i-1}D))= \lambda
N_\CF(x+\CF_{j-1}D), \ \ \lambda\in E.
$$

Finally we extend $N_\CF$ to the whole $\gr^\CF_\bullet D$ by
$E$-linearity. By definition we have either $N(\gr^\CF_i D) = 0$ or
$N(\gr^\CF_i D)= \gr^\CF_j D$ for some $j$.

\begin{defn} For $j\in \{1,\dots, n-1\}$ we say that $j$ is {\it marked} (or a {\it marked index}) for $\CF$
if there is some $i\in \{2, \dots, n\}$ such that $N_\CF(\gr^\CF_i
D)= \gr^\CF_j D$.
\end{defn}

Note that $i$ and $j$ in the above definition are determined by each
other. We write $i=t_\CF(j)$ and $j=s_\CF(i)$.

\begin{rem} We can construct an oriented graph whose vertices are the numbers $1,\dots,
n$; there is an (oriented) edge with source $j$ and terminate $i$ if
and only if $j$ is marked and $i=t_\CF(j)$. The resulting graph
consists of simple vertices and disjointed chains.
\end{rem}

\begin{lem}\label{rem:critical} The following are
equivalent:
\begin{enumerate}
\item\label{it:critical-a} $s$ is marked and $t=t_\CF(s)$.
\item\label{it:critical-b} $N\CF_{t-1}D \cap \CF_s D = N\CF_{t-1}D \cap \CF_{s-1} D$ and
$N\CF_{t}D \cap \CF_s D \supsetneq N\CF_{t}D \cap \CF_{s-1} D$.
\end{enumerate}
\end{lem}
\begin{proof} We have already seen that, if (\ref{it:critical-a}) holds, then (\ref{it:critical-b}) holds.
Conversely, we assume that (\ref{it:critical-b}) holds. Then
$N\CF_{t}D \cap \CF_s D \supsetneq N \CF_{t-1}D\cap \CF_sD$. Thus
$N\CF_{t}D \supsetneq N \CF_{t-1}D$. From $N\CF_{t}D \cap \CF_s D
\supsetneq N \CF_{t-1}D\cap \CF_sD$ we see that there is $x\in
\CF_tD \backslash \CF_{t-1}D$ such that $N(x)\in \CF_sD$. Thus $N
\CF_t D \subseteq N \CF_{t-1}D + \CF_sD$. Next we show that $N\CF_tD
\varsubsetneq N \CF_{t-1}D +\CF_{s-1}D$. If it is not true, then
there exists $y\in \CF_tD \backslash \CF_{t-1}D$ such that $N(y)\in
\CF_{s-1}D$. For every $z\in \CF_tD$, write $z=w+\lambda y$ with
$w\in \CF_{t-1}D$ and $\lambda\in E$. If $N(z)$ is in $\CF_s D$,
then $N(w)$ is also in $\CF_s D$. But $N\CF_{t-1}D\cap \CF_s D
=N\CF_{t-1}D\cap \CF_{s-1}D$. Thus $N(w)$ is in $\CF_{s-1}D$, which
implies that $N(z)=N(w)+\lambda N(y)$ is also in $\CF_{s-1}D$. So,
$N\CF_tD\cap \CF_s D=N\CF_tD\cap \CF_{s-1}D$, a contradiction.
\end{proof}

\begin{defn} We say that an ordered basis $S=\{ e_1,\dots, e_n\}$ of $D$ is
{\it compatible} with $\CF$ if $\CF_r D =\oplus_{i=1}^r Ee_i $ for
all $r \in \{1, \dots, n\} $. If $S=\{e_1, \dots, e_n\}$ is an
ordered basis compatible with $\CF$ and $\varphi(e_i)=\alpha_i e_i$
for every $i\in\{1,\dots, n\}$, we say that $S$ is {\it perfect} for
$\CF$.
\end{defn}

As $\varphi$ is semisimple on $D$, there always exists a perfect
ordered basis for $\CF$.

\begin{lem}\label{lem:perfectbasis}
\begin{enumerate}
\item \label{it:perfectbasis-a} If $s$ is marked for $\CF$ and $t=t_\CF(s)$, then there exists
$e_t\in \CF_tD \backslash \CF_{t-1}D$ such that
$\varphi(e_t)=\alpha_t e_t$ and $N(e_t)\in \CF_{s}D \backslash
\CF_{s-1}D$.
\item \label{it:perfectbasis-b} If $t$ is not $t_\CF(s)$ for any $s$, then there
exists $e_t\in \CF_tD\backslash \CF_{t-1}D$ such that
$\varphi(e_t)=\alpha_t e_t$ and $N(e_t)=0$.
\end{enumerate}
\end{lem}
\begin{proof} Let $\{e'_1, \dots, e'_n\}$ be a perfect basis for
$\CF$.

If $t=t_\CF(s)$, then there exists $x\in \CF_t D \backslash
\CF_{t-1}D$ such that $N(x)\in \CF_sD \backslash \CF_{s-1}D$. Write
$x=\sum_{i=1}^t \lambda_i e'_i$ and put $e_t=\sum_{1\leq i\leq t:
\alpha_i=\alpha_t} \lambda_i e'_i.$ Then $\varphi(e_t)=\alpha_t e_t$
and $N(e_t) \in \CF_{s}D \backslash \CF_{s-1}D$. This proves
(\ref{it:perfectbasis-a}). The proof of (\ref{it:perfectbasis-b}) is
similar.
\end{proof}

\subsection{Strongly marked indices and
$\CL$-invariants}

Assume that $s$ is marked for $\CF$ and $t=t_\CF(s)$. We consider
the decompositions
$$ \CF_t D / \CF_{s-1} D = E \bar{e}_s \oplus L \oplus E \bar{e}_t
$$
that satisfy the following conditions:

$\bullet$ $\overline{\CF}_1 (\CF_t D / \CF_{s-1} D)=E\bar{e}_s$ and
$\overline{\CF}_{t-s}(\CF_tD/\CF_{s-1} D) = E \bar{e}_s\oplus L$,
where $\overline{\CF}$ is the refinement on $\CF_tD/\CF_{s-1} D$
induced by $\CF$.

$\bullet$ Both $L$ and $E \bar{e}_s \oplus E \bar{e}_t$ are stable
by $\varphi$ and $N$; $\varphi(\bar{e}_t)=\alpha_t \bar{e}_t$ and
$N(\bar{e}_t)=\bar{e}_s$.

\noindent Such a decomposition is called an {\it $s$-decomposition.}

\begin{lem} If $s$ is marked, then there exists at least
one $s$-decomposition.
\end{lem}
\begin{proof} By Lemma \ref{lem:perfectbasis} there exists a perfect
basis $\{e_1, \dots, e_n\}$ for $\CF$ such that $N(e_t)=e_s$. For
$i=s, \dots, t$ let $\bar{e}_i$ denote the image of $e_i$ in
$\CF_tD/\CF_{s-1}D$. Then $N(\bar{e}_t)=\bar{e}_s$. Write
$\tilde{L}=\CF_{t-1}D/\CF_{s-1}D$. For every $\alpha\in E$ put
$\tilde{L}^{\alpha}=\{x\in \tilde{L}: \varphi(x)=\alpha x\}$. As $N
\CF_{t-1}D\cap \CF_{s}D=N \CF_{t-1}D\cap \CF_{s-1}D$, we have
$N\tilde{L}^{\alpha_t}\cap E \bar{e}_s=0$. Let $L^{\alpha_s}$ be any
$E$-subspace of $\tilde{L}^{\alpha_s}$ of codimension $1$ that
contains $N\tilde{L}^{\alpha_t}$ and does not contain $E \bar{e}_s$.
Put $L= (\bigoplus\limits_{\alpha\neq \alpha_s}\tilde{L}^{\alpha})
\bigoplus L^{\alpha_s}$. It is easy to verify that $L$ is stable by
$\varphi$ and $N$. Then $\CF_tD/\CF_{s-1}D=E\bar{e}_s \oplus L
\oplus E \bar{e}_t$ is an $s$-decomposition.
\end{proof}

Let $\mathrm{dec}$ denote an $s$-decomposition $ \CF_t D / \CF_{s-1}
D = E \bar{e}_s \oplus L \oplus E \bar{e}_t. $ There are three
possibilities for the filtration on the filtered
$E$-$(\varphi,N)$-submodule $E \bar{e}_s \oplus E \bar{e}_t$:

Case 1. There exist an integer $k'_t>k_s$ and some
$\CL_{\mathrm{dec}}\in E$ (which must be unique) such that
$$\Fil^{i} (E \bar{e}_s \oplus E \bar{e}_t) = \left\{ \begin{array}{ll} E \bar{e}_s\oplus E \bar{e}_t & \text{ if }i \leq k_s, \\
E(\bar{e}_t + \CL_{\mathrm{dec}} \bar{e}_s) & \text{ if } k_s < i\leq k'_t ,\\
0 & \text{ if } i > k'_t.
\end{array}\right.$$

Case 2. There exists an integer $k'_t<k_s$ such that
$$\Fil^{i} (E \bar{e}_s \oplus E \bar{e}_t) = \left\{ \begin{array}{ll} E \bar{e}_s \oplus E \bar{e}_t & \text{ if }i \leq k'_t, \\
E \bar{e}_s & \text{ if } k'_t< i\leq k_s ,\\
0 & \text{ if } i> k_s.
\end{array}\right.$$

Case 3. We have
$$\Fil^{i} (E \bar{e}_s \oplus E \bar{e}_t) = \left\{ \begin{array}{ll} E \bar{e}_s \oplus E \bar{e}_t  & \text{ if }i \leq k_s, \\
0 & \text{ if } i>k_s.
\end{array}\right.$$

Similarly, there are three possibilities for the filtration on the
quotient of $\CF_tD/\CF_{s-1}D$ by $L$. Below we will denote the
images of $\bar{e}_s$ and $\bar{e}_t$ in $(\CF_tD/\CF_{s-1}D)/L$ by
the original notations $\bar{e}_s$ and $\bar{e}_t$.

Case $1'$. There exist an integer $k'_s < k_t$ and some
$\CL'_{\mathrm{dec}}\in E$ (which must be unique) such that
$$\Fil^{i} (E \bar{e}_s \oplus E \bar{e}_t) = \left\{ \begin{array}{ll} E \bar{e}_s\oplus E \bar{e}_t & \text{ if }i \leq k'_s, \\
E(\bar{e}_t + \CL'_{\mathrm{dec}} \bar{e}_s) & \text{ if } k'_s < i\leq k_t ,\\
0 & \text{ if } i > k_t.
\end{array}\right.$$

Case $2'$. There exists an integer $k'_s>k_t$ such that
$$\Fil^{i} (E \bar{e}_s \oplus E \bar{e}_t) = \left\{ \begin{array}{ll} E \bar{e}_s \oplus E \bar{e}_t & \text{ if }i \leq k_t, \\
E \bar{e}_s & \text{ if } k_t< i\leq k'_s ,\\
0 & \text{ if } i> k'_s.
\end{array}\right.$$

Case $3'$. We have
$$\Fil^{i} (E \bar{e}_s \oplus E \bar{e}_t)= \left\{ \begin{array}{ll} E \bar{e}_s \oplus E \bar{e}_t & \text{ if }i \leq k_t, \\
0 & \text{ if } i>k_t.
\end{array}\right.$$

If Case 1 and Case $1'$ happen, we always have $k_s\leq k'_s$ and
$k'_t\leq k_t$. If further $k'_s < k'_t$ (which happens only when
$k_s<k_t$), we say that $\mathrm{dec}$ is a {\it perfect
$s$-decomposition} (for $\CF$). In this case we must have
$\CL_{\mathrm{dec}}=\CL'_{\mathrm{dec}}$.

\begin{defn}\label{defn:Fontaine-Mazur}
If there exists a perfect $s$-decomposition, we say that $s$ is {\it
strongly marked} (or a {\it strongly marked index}). In this case we
attached to $s$ an invariant $\CL_{\mathrm{dec}}$, where
$\mathrm{dec}$ is a perfect $s$-decomposition. Proposition
\ref{prop:fon-mazur-inv} below tells us that $\CL_{\mathrm{dec}}$ is
independent of the choice of perfect $s$-decompositions. We denote
it by $\CL_{\CF,s}$ and call it the {\it Fontaine--Mazur
$\CL$-invariant} associated to $(\CF,s)$.
\end{defn}

\begin{prop}
In the case of $t=s+1$, $s$ is strongly marked if and only if
$k_s<k_t$.
\end{prop}
\begin{proof}
We have already seen the necessity of $k_s<k_t$.

In the case of $t=s+1$, as $L=0$, both the filtered
$E$-$(\varphi,N)$-submodule $E\bar{e}_s\oplus E\bar{e}_t$ and the
quotient $(\CF_tD/\CF_{s-1}D)/L$ coincide with $\CF_tD/\CF_{s-1}D$.
Since the Hodge--Tate weight of $\mathrm{gr}^\CF_t D$ is $k_t$, and
since as a filtered $E$-$(\varphi,N)$-module $\mathrm{gr}^\CF_t D$
is the quotient of $\CF_tD/\CF_{s-1}D$ by $\mathrm{gr}^\CF_s D =
E\bar{e}_s$, there exists some $\CL\in E$ such that
$$\bar{e}_t+ \CL \bar{e}_s \in \Fil^{k_t} (\CF_tD/\CF_{s-1}D),\ \
\bar{e}_t+ \CL \bar{e}_s \notin \Fil^{k_t+1} (\CF_tD/\CF_{s-1}D).
$$ Since the Hodge--Tate weight of
$\mathrm{gr}^\CF_s D$ is $k_s$, we have
$$\bar{e}_s \in \Fil^{k_s} (\CF_tD/\CF_{s-1}D),  \ \ \bar{e}_s  \notin \Fil^{k_s+1}
(\CF_tD/\CF_{s-1}D). $$ Therefore, when $k_s<k_t$, Case $1$ and Case
$1'$ happen ($k'_t=k_t$ and $k'_s=k_s$).
\end{proof}

\begin{prop} \label{prop:fon-mazur-inv}
If $\mathrm{dec}_1$ and $\mathrm{dec}_2$ are two perfect
$s$-decompositions, then
$\CL_{\mathrm{dec}_1}=\CL_{\mathrm{dec}_2}$.
\end{prop}
\begin{proof} Without loss of generality we assume that $\bar{e}_s$ in the two perfect $s$-decompositions are same. Let $k^{(1)}_s$, $k^{(1)}_t$, $L^{(1)}$ and $\bar{e}_t^{(1)}$ (resp.
$k^{(2)}_s$, $k^{(2)}_t$, $L^{(2)}$ and $\bar{e}_t^{(2)}$) denote
$k'_s$, $k'_t$, $L$ and $\bar{e}_t$ for $\mathrm{dec}_1$ (resp.
$\mathrm{dec}_2$). We assume that $k^{(1)}_s\leq k^{(2)}_s$.

As $N(\bar{e}^{(2)}_t-\bar{e}^{(1)}_t)=0$, we have
$$\bar{e}^{(2)}_t-\bar{e}^{(1)}_t\in
(\CF_{t-1}D/\CF_{s-1}D)^{\varphi=\alpha_t} =
(L^{(1)})^{\varphi=\alpha_t}.$$ Thus $\bar{e}^{(2)}_t$ and
$\bar{e}^{(1)}_t$ have the same image in
$(\CF_tD/\CF_{s-1}D)/L^{(1)}$. We will denote the images of $e_s$,
$\bar{e}_t^{(1)}$ and $\bar{e}_t^{(2)}$ in
$(\CF_tD/\CF_{s-1}D)/L^{(1)}$ by the original notations.

As $\bar{e}^{(2)}_t + \CL_{\mathrm{dec}_2} \bar{e}_s$ is in
$\Fil^{k^{(2)}_t}(\CF_tD/\CF_{s-1}D)$, and as the Hodge--Tate
weights $k_s^{(1)}$ and $k_t$ of $(\CF_tD/\CF_{s-1}D)/L^{(1)}$
satisfy $k_s^{(1)} \leq k_s^{(2)} < k^{(2)}_t \leq k_t$, we have
\begin{eqnarray*}\Fil^{k_t}\Big((\CF_tD/\CF_{s-1}D)/L^{(1)}\Big) &=&
\Fil^{k^{(2)}_t}\Big((\CF_tD/\CF_{s-1}D)/L^{(1)}\Big) \\ &=& E
(\bar{e}^{(2)}_t +\CL_{\mathrm{dec}_2} \bar{e}_s)= E
(\bar{e}^{(1)}_t +\CL_{\mathrm{dec}_2} \bar{e}_s), \end{eqnarray*}
which implies that $\CL_{\mathrm{dec}_1}=\CL_{\mathrm{dec}_2}$.
\end{proof}

\begin{defn}\label{defn:s-perfect-basis} Let $s$ be strongly marked. We say that a perfect basis
$\{e_1, \dots, e_n\}$ for $\CF$ is {\it $s$-perfect} if it satisfies
the following conditions:

$\bullet$ $E \bar{e}_s \bigoplus \Big(\bigoplus_{s<i<t} E
\bar{e}_i\Big) \bigoplus E \bar{e}_t$ is a perfect $s$-decomposition
where $\bar{e}_i$ is the image of $e_i$ in $D/\CF_{s-1}D$,

$\bullet$ $N(e_t)=e_s$.

$\bullet$ For every $i>t_\CF(s)$ writing
$N(e_i)=\sum_{j=1}^{i-1}\lambda_{i,j} e_j$ we have
$\lambda_{i,s}=0$.
\end{defn}

\begin{rem}\label{rem:s-perfect-basis} The first condition in Definition
\ref{defn:s-perfect-basis} implies that, for every $i=s+1, \dots,
t_\CF(s)-1$ if we write $N(e_i)=\sum_{j=1}^{i-1}\lambda_{i,j} e_j$,
then $\lambda_{i,s}=0$.
\end{rem}

\begin{lem}\label{lem:s-perfect-basis} If $s$ is strongly marked, then there exists an $s$-perfect
basis.
\end{lem}
\begin{proof} Write $t=t_\CF(s)$.
Let $\{e_1, \dots, e_{s-1}\}$ be a perfect ordered basis of
$\CF_{r-1}D$.

As $s$ is strongly marked, there exists a perfect $s$-decomposition
$\CF_tD/\CF_{s-1}D = E \bar{e}_s \oplus L \oplus E \bar{e}_t$.
Choose a perfect basis $\{\bar{e}_i: s<i<t\}$ for the induced
refinement on $L$ (identified with $\CF_{t-1}D/\CF_sD$). For $i\in
\{s+1, \dots, t\}$ let $e_i\in \CF_t D$ be any lifting of
$\bar{e}_i$ such that $\varphi(e_i)=\alpha_i e_i$. Put $e_s=N(e_t)$.

For every $i>t$ there exists $e'_i\in \CF_iD \backslash \CF_{i-1}D$
such that $\varphi(e'_i)=\alpha_ie'_i$. Next we define $e_i$ for
$i>t$ recursively from $t+1$ to $n$. Write
$N(e'_i)=\sum_{j=1}^{i-1}\mu_{i,j}e_j$. If $\mu_{i,s}=0$, we put
$e_i=e'_i$. If $\mu_{i,s}\neq 0$, then $\alpha_i=\alpha_t$. In this
case we put $e_i=e'_i- \mu_{i,s}e_t$. It is easy to prove the
resulting ordered basis $\{e_1, \dots, e_n\}$ is $s$-perfect.
\end{proof}

\subsection{Duality}

Let $D$ be a filtered $E$-$(\varphi,N)$-module, with $\CF$ a
refinement on $D$.

Let $D^*$ the filtered $E$-$(\varphi, N)$-module that is the dual of
$D$. Let $\check{\CF}$ be the refinement on $D^*$ such that
$$ \check{\CF}_i D^* : = (\CF_{n-i}D)^\bot=\left\{y\in D^*: \langle y, x
\rangle =0 \text{ for all }x\in \CF_{n-i}D \right\}. $$ We call
$\check{\CF}$ the {\it dual refinement} of $\CF$.

\begin{prop}\label{prop:duality-critical} If $s$ is marked for $\CF$ and $t=t_\CF(s)$, then
$n+1-t$ is marked for $\check{\CF}$ and
$n+1-s=t_{\check{\CF}}(n+1-t)$.
\end{prop}
\begin{proof} By Lemma \ref{rem:critical} we only need to prove that
\begin{equation}\label{eq:duality-a} N \check{\CF}_{n-s}\cap \check{\CF}_{n+1-t} = N
\check{\CF}_{n-s} \cap \check{\CF}_{n-t} \end{equation} and
\begin{equation}\label{eq:duality-b} N \check{\CF}_{n+1-s}\cap \check{\CF}_{n+1-t} \supsetneq N
\check{\CF}_{n+1-s} \cap \check{\CF}_{n-t}.
\end{equation}

For (\ref{eq:duality-a}) we have
\begin{eqnarray*} N \check{\CF}_{n-s}\cap \check{\CF}_{n+1-t} &=&
\{N (y^*): y^* \in \check{\CF}_{n-s}, \langle N (y^*), x\rangle=0 \
\forall\ x\in \CF_{t-1} \} \\
& = & \{N (y^*): y^* \in \check{\CF}_{n-s}, \langle  y^*, N(x)
\rangle=0 \ \forall\ x\in \CF_{t-1} \} \\
&=& N( ( \CF_{s}+ N \CF_{t-1})^\bot)
\end{eqnarray*} and
$$  N \check{\CF}_{n-s}\cap \check{\CF}_{n-t} = N( ( \CF_{s}+ N \CF_{t})^\bot).
$$ Then (\ref{eq:duality-a}) follows from the relation $\CF_{s}+ N \CF_{t-1}=\CF_s+ N
\CF_t$.

For (\ref{eq:duality-b}) we have
\begin{eqnarray*}
(N \check{\CF}_{n+1-s})^\bot &=& \{ x\in D: \langle N(y^*),
x\rangle=0 \ \forall\ y\in \check{\CF}_{n+1-s} \} \\
&=& \{ x\in D: \langle y^*, N(x)\rangle=0 \ \forall\ y\in
\check{\CF}_{n+1-s} \} \\ &=& \{ x\in D: N (x)\in \CF_{s-1} \}.
\end{eqnarray*} Thus
$$ N \check{\CF}_{n+1-s}\cap \check{\CF}_{n+1-t} =(\{x\in D: N(x) \in \CF_{s-1}\}+ \CF_{t-1})^\bot
$$ and
$$N \check{\CF}_{n+1-s}\cap \check{\CF}_{n-t} =(\{x\in D: N(x) \in \CF_{s-1}\}+ \CF_t)^\bot.
$$ But $\{x\in D: N(x) \in \CF_{s-1}\}+ \CF_{t} \supsetneq \{x\in D: N(x) \in \CF_{s-1}\}+
\CF_{t-1}$. Indeed, \begin{eqnarray*} && (\{x\in D: N(x) \in
\CF_{s-1}\}+ \CF_{t}) / (\{x\in D: N(x) \in \CF_{s-1}\}+ \CF_{t-1}) \\
&& \cong\CF_t/ (\CF_{t-1}+ \{x\in \CF_t: N(x)\in \CF_{s-1}\} )
=\CF_t/\CF_{t-1}.
\end{eqnarray*} Thus (\ref{eq:duality-b}) follows.
\end{proof}

If $L\subset M$ are submodules of $D$, then $M^\bot\subset L^\bot$.
The pairing $\langle \cdot, \cdot \rangle:  L^\bot\times M$ induces
a non-degenerate pairing on $L^\bot/M^\bot \times M/L  $, so that we
can identify $L^\bot/M^\bot$ with the dual of $M/L$ naturally. In
particular, $\gr^{\check{\CF}}_i D^*$ is naturally isomorphic to the
dual of $\gr^{\CF}_{n+1-i}D$. Thus $\gr^{\check{\CF}}_\bullet D^*$
is naturally isomorphic to the dual of $\gr^\CF_\bullet D$.

\begin{prop} $N_{\check{\CF}}$ is dual to $-N_\CF$.
\end{prop}
\begin{proof}
By Proposition \ref{prop:duality-critical}, $N_\CF(\gr^\CF_t
D)=\gr^\CF_s D$ if and only if
$N_{\check{\CF}}(\gr^{\check{\CF}}_{n+1-s}D^*)=\gr^{\check{\CF}}_{n+1-t}
D^*$. We choose a perfect basis $\{e_1,\dots e_n\}$ for $\CF$ such
that $N(e_t)=e_s$. Then $\{e_i+\CF_{i-1}D: i=1,\dots, n\}$ is a
basis of $\mathrm{gr}^{\CF}_\bullet D$, and its dual basis is
$\{e^*_i+\check{\CF}_{n-i}D: i=1,\dots, n\}$.

Note that $N_\CF(e_t+\CF_{t-1}D)=e_s+\CF_{s-1}D$. What we need to
show is that $N_{\check{\CF}}(e^*_s + \check{\CF}_{n-s}D^*)=
-e^*_t+\check{\CF}_{n-t}D^*$. For this we only need to prove that
$Ne^*_s+e^*_t$ is in $\check{\CF}_{n-t}D^* + N
\check{\CF}_{n-s}D^*$. We have
\begin{eqnarray*}(\check{\CF}_{n-t}D^* + N
\check{\CF}_{n-s}D^*)^\bot & = & (\check{\CF}_{n-t}D^*)^\bot \cap (
N \check{\CF}_{n-s}D^*)^\bot  \\
& = & \CF_t D \cap \{ x\in D: N(x)\in \CF_s D\} \\ & = & \{x\in
\CF_t D: N(x)\in \CF_s D\}.
\end{eqnarray*} For every $x\in \CF_tD$ such that $N(x)\in \CF_s D$,
we can write $x$ in the form $x=ae_t+y$ with $y\in \CF_{t-1}D$. Then
$\langle e^*_t, y\rangle=0$. As $N(y)\in \CF_{s}D \cap N \CF_{t-1}D=
\CF_{s-1}D \cap N \CF_{t-1}D$, we have $\langle e^*_s,
N(y)\rangle=0$. Hence
$$ \langle Ne^*_s + e^*_t  ,  x\rangle = \langle e^*_s, -N(ae_t+y) \rangle + \langle e^*_t, ae_t+y \rangle =0
,$$ as expected.
\end{proof}

Let $\{e_1, \dots, e_n\}$ be a perfect basis for $\CF$, and let
$\{e^*_1, \dots, e^*_n\}$ be the dual basis of $\{e_1,\dots, e_n\}$.
Then $\{e^*_n ,\dots, e^*_1\}$ is perfect for $\check{\CF}$.

\begin{prop}\label{prop:duality-strong-critical}
\begin{enumerate}
\item\label{it:duality-strong-critical-a} $s$ is strongly marked for $\CF$ if and only if
$n+1-t_\CF(s)$ is strongly marked for $\check{\CF}$.
\item\label{it:duality-strong-critical-b} If $s$ is strongly marked for $\CF$, then
$\{e_1,\dots, e_n\}$ is $s$-perfect for $\CF$ if and only if
$\{e^*_n, \dots e^*_{s+1}, -e^*_s, e^*_{s-1}, \dots, e^*_1\}$ is
$(n+1-t_\CF(s))$-perfect for $\check{\CF}$.\end{enumerate}
\end{prop}
\begin{proof} Assume that $s$ is strongly marked for $\CF$, $t=t_\CF(s)$ and $\{e_1, \dots,
e_n\}$ is $s$-perfect for $\CF$. Let $\bar{e}_i$ $(s\leq i\leq t)$
be the image of $e_i$ in $\CF_tD/\CF_{s-1}D$, and put
$L=\bigoplus_{s<i<t} E \bar{e}_i$. By the definition of $s$-perfect
bases, $E \bar{e}_s \bigoplus L \bigoplus E \bar{e}_t$ is a perfect
$s$-decomposition.

Similarly, let $\bar{e}^*_i$ $(s\leq i\leq t)$ be the image of
$e^*_i$ in $\check{\CF}_{n+1-s}D^*/\check{\CF}_{n-t}D^*$. Note that
$\check{\CF}_{n+1-s}D^*/\check{\CF}_{n-t}D^*$ is naturally
isomorphic to the dual of $\CF_t D/\CF_{s-1}D$. Put
$\check{L}=\bigoplus_{s<i<t} E \bar{e}^*_i$. Then $\check{L}=(E
\bar{e}_s\oplus E\bar{e}_t)^\bot$ and $L=(E\bar{e}^*_t \oplus E
\bar{e}^*_s)^\bot$. Note that $E \bar{e}^*_t \oplus E \bar{e}^*_s$
is isomorphic to the dual of the quotient of $\CF_tD/\CF_sD$ by $L$,
and the quotient of $\check{\CF}_{n+1-s}D^*/\check{\CF}_{n-t}D^*$ by
$\check{L}$ is isomorphic to the dual of $E e_s \oplus E e_t$. Hence
$E \bar{e}_t \oplus \check{L} \oplus E \bar{e}_s$ is an
$(n+1-t_\CF(s))$-perfect decomposition for $\check{\CF}$. This
proves (\ref{it:duality-strong-critical-a}).

For $i<s$, write $N(e^*_i)=\sum_{j=i+1}^n \lambda_{i,j} e^*_j$. Then
$$\lambda_{i,t} = \langle N (e^*_i), e_t\rangle = \langle e^*_i, -
N(e_t)\rangle = \langle e^*_i, -e_s \rangle=0.$$ Write
$N(-e^*_s)=\sum_{j=t}^n \lambda_{s,j} e^*_j$. Then
$$\lambda_{s,j} = \langle N (-e^*_s), e_j\rangle = \langle e^*_s,
N(e_j)\rangle =\left\{\begin{array}{ll} 1 & \text{ if }j=t, \\ 0 &
\text{ if }j > t.
\end{array}\right.$$ Thus $N(-e^*_s)=e^*_t$. This proves (\ref{it:duality-strong-critical-b}).
\end{proof}

\section{Galois cohomology of $V^*\otimes_E V$} \label{sec:main-teck}

\subsection{A lemma}
Let $\CL$ be an element in $E$. Let $D$ be a filtered
$E$-$(\varphi,N)$-module $D=E f_1 \oplus E f_2 \oplus E f_3$ with
\begin{eqnarray*}&& \varphi(f_1)=p^{-1}f_1, \ \varphi(f_2)=f_2, \ \varphi(f_3)=f_3,
\\ && N(f_1)=0, \ N(f_2)= - f_1, \ N(f_3)=f_1, \\
&& \Fil^0 D = E (f_2- \CL f_1) \oplus E (f_3+\CL f_1).
\end{eqnarray*}
Let $\pi_i$ be the projection map
$$ \BX_\st(D)
\rightarrow \BB_{\st,E}, \ \ \ \sum_{j=1}^3 a_j f_j \mapsto a_j. $$

\begin{lem}\label{lem:aux} Let $c: G_{\BQ_p}\rightarrow \BX_\st(D)$ be a
$1$-cocycle whose class in $H^1(\BX_\st(D))$ belongs to
$\ker(H^1(\BX_\st(D))\rightarrow H^1(\BX_\dR(D)))$. Then there exist
$\gamma_{2,1}, \gamma_{2,2}, \gamma_{3,1}, \gamma_{3,2}\in E$ such
that $\pi_2(c)=\gamma_{2,1}\psi_1+\gamma_{2,2}\psi_2$ and
$\pi_3(c)=\gamma_{3,1}\psi_1+\gamma_{3,2}\psi_2$. Furthermore,
$\gamma_{2,1}-\gamma_{3,1}=\CL(\gamma_{2,2}-\gamma_{3,2})$.
\end{lem}

The proof of Lemma \ref{lem:aux} needs the Tate duality pairing $
H^1(\BQ_p)\times H^1(\BQ_p(1))\rightarrow H^2(\BQ_p(1)) $. We give a
precise description of it following \cite[\S 4.1]{Cz2010}.

Let $v\in \BB_{\cris}^{\varphi=p}$ be such that $v/t_\cyc\in
\BB^{\varphi=1}_{\cris}$ and $1/t_\cyc$ have the same image in
$\BB_{\dR}/\BB^+_{\dR}$. Let $u$ be the element of $\BB_\st$ such
that $u\in \Fil^1 \BB_{\dR}$, $\varphi(u)=pu$, $N(u)=-1$, and
$\sigma(u)-u\in \BQ_p \, t_\cyc$. Then $\sigma\mapsto \sigma(u)-u$
and $\sigma\mapsto \sigma(v)-v$ form an $E$-basis of $H^1(E(1))$.
Let $(b_1,b_2)\in E\times E$ denote the $1$-cocycle $\sigma\mapsto
(\sigma-1)(b_1 u+ b_2 v)$. The $E$-representation corresponding to
$(1, \ell)$ is attached to the filtered $E$-$(\varphi, N)$-module
$(D_\ell = Ee \oplus E f , \varphi,N,\Fil)$ with
$$ \varphi( e ) = p^{-1}e, \ \ \varphi(f) = f, \ \ N(e)=0, \ \
N(f)=e
$$ and
$$ \Fil^j D_\ell =\left\{ \begin{array}{ll} D_\ell & \text{ if } j \leq -1 , \\
E\cdot (f+\ell e) & \text{ if } j=0 ,  \\ 0 & \text{ if } j\geq 1 .
\end{array} \right. $$

Let $H^1(E)\times H^1(E(1)) \rightarrow E$ be the pairing induced by
the Tate duality pairing
$$ H^1(\BQ_p)\times H^1(\BQ_p(1))\rightarrow
H^2(\BQ_p(1)) $$ and the isomorphism $H^2(\BQ_p(1))\cong \BQ_p$ from
local class field theory. Then precisely we have
$$ <a_1 \psi_1+a_2\psi_2, (b_1,b_2)> = -a_1b_1+a_2b_2 . $$

\vskip 10pt

\noindent{\it Proof of Lemma \ref{lem:aux}.} Write
$c_\sigma=\lambda_{1,\sigma}f_1+\lambda_{2,\sigma}f_2+\lambda_{3,\sigma}f_3$.
As $c$ takes values in $\BX_\st(D)$, we have $\lambda_{2,\sigma},
\lambda_{3,\sigma}\in E$, $\lambda_{1,\sigma}\in U_{1,1}$, and
$N(\lambda_{1,\sigma})=\lambda_{3,\sigma}-\lambda_{2,\sigma}$. This
ensures the existence of $\gamma_{2,1}, \gamma_{2,2},
\gamma_{3,1},\gamma_{3,2}$.

To show that
$\gamma_{2,1}-\gamma_{3,1}=\CL(\gamma_{2,2}-\gamma_{3,2})$, we
define a new filtration $Fil$ on $D$ by
$$ Fil^i (D) = \left\{  \begin{array}{ll} D & \text{ if } i\leq
-1 ,
\\  E (f_2- \CL f_1) \oplus E (f_3+\CL f_1) & \text{ if } i=0 , \\
0 & \text{ if } i\geq 1 .
 \end{array} \right. $$
Then $Fil \approx \Fil$ and $(D, Fil)$ is admissible. Let $W$ be the
semistable representation of $G_{\BQ_p}$ attached to $D_W=(D, Fil)$.

As $Fil\approx\Fil$, by proposition \ref{prop:same-kernel}, $[c]$ is
in the kernel of $ H^1(\BX_\st(D_W))\rightarrow H^1(\BX_\dR(D_W))$.
By the exact sequence
$$ \xymatrix{ H^1(W) \ar[r] & H^1(\BX_\st(D_W)) \ar[r] & H^1(\BX_\dR(D_W)) }
$$ there exists a $1$-cocycle $c^{(1)}: G_{\BQ_p}\rightarrow
W$ such that the image of $[c^{(1)}]$ by $H^1(W)\rightarrow
H^1(\BX_\st(D_W))$ is $[c]$.

Observe that the filtered $E$-$(\varphi,N)$-submodule of $D_W$
generated by $f_1$ (resp. by $f_2 + f_3$) is admissible and thus
comes from an $E$-subrepresentation of $W$, denoted by $W_0$ (resp.
$W'$). Let $W_1$ be the quotient of $W$ by $W'$, $\pi_{W,W_1}$ the
map $W\rightarrow W_1$. Then $W_0$ injects to $W_1$. The image of
$W_0$ in $W_1$ is again denoted by $W_0$ by abuse of notation. The
quotients of $W$ and $W_1$ by $W_0$ are denoted by $T$ and $T_1$
respectively. Then we have the following commutative diagram
\begin{equation}\label{eq:comm-diag-galois-repn}
\xymatrix{   0 \ar[r] & W_0 \ar[r] \ar@{=}[d] & W
\ar[r]\ar[d]^{\pi_{W,W_1}} & T \ar[r]\ar[d] & 0
\\  0 \ar[r] & W_0 \ar[r] & W_1 \ar[r] & T_1 \ar[r] &
0, }\end{equation} where the horizontal lines are exact.

We compute the image of $[c^{(1)}]$ by the map $H^1(W)\rightarrow
H^1(T)$. Note that the action of $G_{\BQ_p}$ on $T$ is trivial. So
we may identify $T$ with $D_T$, the filtered
$E$-$(\varphi,N)$-module attached to $T$. We consider the
commutative diagram
\[ \xymatrix{ H^1(W)\ar[d] \ar[r] & H^1(\BX_\st(D_W))\ar[d] \ar[r] & H^1(\BX_\dR(D_W))\ar[d] \\
H^1(T) \ar[r] & H^1(\BX_\st(D_T))\ar[r] & H^1(\BX_\dR(D_T))}
\] where the horizontal lines are exact. As the image of $[c^{(1)}]$ in
$H^1(\BX_\st(D_W))$ is $[c]$, the image of $[c^{(1)}]$ in
$H^1(\BX_\st(D_T))$ by the map $H^1(W)\rightarrow H^1(T)\rightarrow
H^1(\BX_\st(D_T))$ coincides with the class of the $1$-cocyle
$\sigma \mapsto
\lambda_{2,\sigma}\bar{f}_2+\lambda_{3,\sigma}\bar{f}_3$, where
$\bar{f}_2$ and $\bar{f}_3$ are the images of $f_2$ and $f_3$ in
$D_T$ respectively. As $H^1(T)\rightarrow H^1(\BX_\st(D_T))$ is an
isomorphism by Remark \ref{rem:semi-st-ht-pos}, the image of $[c]$
in $H^1(T)$ coincides with the class of $\sigma \mapsto
\lambda_{2,\sigma}\bar{f}_2+\lambda_{3,\sigma}\bar{f}_3$, where
$\{\bar{f}_2,\bar{f}_3\}$ is considered as an $E$-basis of $T$.

Write $c^{(2)}$ for the $1$-cocycle
$$G_{\BQ_p}\xrightarrow{c^{(1)}} W\rightarrow T\rightarrow T_1.$$ As $T_1$
is the quotient of $T$ by $E (\bar{f}_2+\bar{f}_3)$ we have
$$[c^{(2)}]= [(\lambda_{2}- \lambda_{3})\bar{\bar{f}}_2]  =  [\Big((\gamma_{2,1}
-\gamma_{3,1})\psi_1 +(\gamma_{2,2}-\gamma_{3,2})\psi_2
\Big)\bar{\bar{f}}_2]
$$ where $\bar{\bar{f}}_2$ is the image of $\bar{f}_2$ in $T_1$.

From the diagram (\ref{eq:comm-diag-galois-repn}) we obtain the
following commutative diagram
$$ \xymatrix{ H^1(W) \ar[r]\ar[d]^{\pi_{W,W_1}} & H^1(T) \ar[r]\ar[d] & H^2(W_0)\ar@{=}[d] \\
H^1(W_1)\ar[r] & H^1(T_1) \ar[r] & H^2(W_0)  , } $$ where the
horizontal lines are exact. Note that $T_1$ is isomorphic to $E$,
and $W_0$ is isomorphic to $E(1)$. Being the image of
$[\pi_{W,W_1}(c^{(1)})]$ in $H^1(T_1)$, $[c^{(2)}]$ lies in the
kernel of $H^1(T_1)\rightarrow H^2(W_0)$. As an extension of $E$ by
$E(1)$, $W_1$ corresponds to the element $(1,\CL)\in H^1(E(1))$. So
the map $H^1(T_1)\rightarrow H^2(W_0)=E$ is given by
$$(a\psi_1+b\psi_2)\bar{\bar{f}}_2\mapsto (a\psi+b\psi_2)\cup (1, \CL) =
- a + b\ \CL.$$ This implies that
$\gamma_{2,1}-\gamma_{3,1}=\CL(\gamma_{2,2}-\gamma_{3,2})$. \qed

\subsection{$1$-cocycles with values in $V^*\otimes_EV$ and
$\CL$-invariants}

Let $D$ be a (not necessarily admissible) filtered
$E$-$(\varphi,N)$-module with a refinement $\CF$. Suppose that
$\varphi$ is semisimple on $D$. Let $\alpha_1,\dots, \alpha_n$ be
the ordering of eigenvalues of $\varphi$ and let $k_1, \dots, k_n$
be the ordering of Hodge--Tate weights fixed by $\CF$. Let $s\in\{1,
\dots, n-1\}$ be strongly marked for $\CF$ and $t=t_\CF(s)$. Let
$\{e_1,\dots, e_n\}$ be an $s$-perfect basis for $\CF$.

Let $D^*$ be the filtered $E$-$(\varphi, N)$-module that is the dual
of $D$, $\{e^*_1, \dots, e^*_n\}$ the dual basis of $\{e_1,\dots,
e_n\}$. Let $\check{\CF}$ be the dual refinement of $\CF$. By
Proposition \ref{prop:duality-strong-critical}, $n+1-t$ is strongly
marked for $\check{\CF}$, $t_{\check{\CF}}(n+1-t)=n+1-s$, and
$\{e^*_n, \dots e^*_{s+1}, -e^*_s, e^*_{s-1}, \dots, e^*_1\}$ is
$(n+1-t)$-perfect for $\check{\CF}$.

As $\{e_1, \dots, e_n\}$ is $s$-perfect for $\CF$,
$$(\bigoplus_{i<s}Ee_i)\bigoplus E e_s \bigoplus E e_t $$ is stable
by $\varphi$ and $N$, and let $D_1$ denote this filtered
$E$-$(\varphi,N)$-submodule of $D$;
$$(\bigoplus_{i<s}Ee_i)\bigoplus (\bigoplus_{s<i<t}Ee_i)$$ is also
stable by $\varphi$ and $N$, and let $\overline{D}_2$ be the
quotient of $D$ by this filtered $E$-$(\varphi,N)$-submodule.
Similarly, as $\{e^*_n, \dots e^*_{s+1}, -e^*_s, e^*_{s-1}, \dots,
e^*_1\}$ is $(n+1-t)$-perfect for $\check{\CF}$,
$$ (\bigoplus_{j > t} E e^*_j ) \bigoplus (\bigoplus_{t>j>s}Ee^*_j)
$$ is stable by $\varphi$ and $N$. The quotient of $D^*$ by this
filtered $E$-$(\varphi,N)$-submodule is naturally isomorphic to the
dual of $D_1$, so we write $D^*_1$ for this quotient.

Put $I=\{s\}\cup \{i\in \BZ: t\leq i \leq n\}$ and $J=\{t\}\cup
\{j\in\BZ: 1\leq j\leq s\}$. By abuse of notations, let $e_i$ ($i\in
I$) denote its image in $\overline{D}_2$; similarly let $e^*_j$
($j\in J$) denote its image in $D_1^*$. Let $\mathscr{D}$ be
filtered $E$-$(\varphi,N)$-module $D_{1}^* \otimes_E
\overline{D}_2$. The image of $e^*_j\otimes e_i\in D^*\otimes_E D$
($i\in I, j\in J$) in $\mathscr{D}$ will be denoted by $e^*_j\otimes
e_i$ again since this makes no confusion.

For $e\otimes e^* \in D_{1}\otimes_E
\overline{D}_2^*=\mathscr{D}^*$, let $\pi_{e\otimes e^*}$ be the
$G_{\BQ_p}$-equivariant map
$$\BX_{\st}(\mathscr{D})\rightarrow \BB_{\st,E} \ \ \ x\mapsto
<e\otimes e^*, x>.$$ We write $\pi_{j,i}$ $(i\in I, j\in J)$ for
$\pi_{e_j\otimes e^*_i}$. Then $\pi_{j,i}$ is induced from the
($G_{\BQ_p}$-equivariant) projection map
$$ \BB_{\st,E}\otimes_E \mathscr{D} \rightarrow \BB_{\st,E} , \ \ \sum_{h\in J}\sum_{\ell\in I} b_{h,\ell} e^*_{h}\otimes e_{\ell} \mapsto b_{j,i}.
$$ The morphism $\pi_{j,i}$ ($i\in I, j\in J$) induces a morphism $H^1(\BX_\st(\mathscr{D}))\rightarrow
H^1(\BB_{\st,E})$ again denoted by $\pi_{j,i}$.

Let $\mu_s$ be the minimal integer such that
$N^{\mu_s+1}(e_s^*\otimes e_s)=0$. We define $\mu_t$ similarly. By
Lemma \ref{lem:Bst-proj} (\ref{it:Bst-proj-b}) the image of
$\pi_{s,s}$ is in $U_{\mu_s, 0}$ and the image of $\pi_{t,t}$ is in
$U_{\mu_t,0}$.

\begin{thm}\label{prop:main-use}
Let $c:G_{\BQ_p}\rightarrow \BX_\st(\mathscr{D})$ be a $1$-cocycle.
\begin{enumerate}
\item \label{it:main-use-a} If $\pi_{j,s}([c])=0$ for every $j<s$ and if $\pi_{s,i} ([c])=0$ for every $i\geq t$, then there exist
$x_s\in U_{\mu_s,0}$ and $\gamma_{s,1}$, $\gamma_{s,2}\in E$ such
that
$$\pi_{s,s}(c_\sigma)=\gamma_{s,1} \psi_1(\sigma)+\gamma_{s,2}
\psi_2(\sigma)+ (\sigma-1)x_s .$$
\item \label{it:main-use-b} If $\pi_{j,t}([c])=0$ for every $j\leq s$ and if
$ \pi_{t,i}([c])=0 $ for every $i>t$, then there exist $x_t\in
U_{\mu_t,0}$ and $\gamma_{t,1}$, $\gamma_{t,2}\in E$ such that
$$\pi_{t,t}(c_\sigma)=\gamma_{t,1} \psi_1(\sigma)+\gamma_{t,2}
\psi_2(\sigma)+ (\sigma-1)x_t .$$
\item \label{it:main-use-c}
If the conditions in $(\mathrm{\ref{it:main-use-a}})$ and
$(\mathrm{\ref{it:main-use-b}})$ hold, and if $[c]$ belongs to
$\ker(H^1(\BX_\st(\mathscr{D}))\rightarrow
H^1(\BX_\dR(\mathscr{D})))$,  then
$$ \gamma_{s,1}-\gamma_{t,1} = \CL_{\CF,s}
(\gamma_{s,2}-\gamma_{t,2}),
$$ where $\CL_{\CF,s}$ is the $\CL$-invariant defined in Definition \ref{defn:Fontaine-Mazur}. \end{enumerate}
\end{thm}
\begin{proof} By Remark \ref{rem:s-perfect-basis} if $i<t$, then $ <Ne^*_s, e_i> =-<e^*_s, N e_i> =0
$. So $Ne^*_s$ is an $E$-linear combination of $e^*_i$ ($i\geq t$).
On the other hand, $Ne_s$ is an $E$-linear combination of $e_j$
$(j<s)$. Thus $N(e_s\otimes e^*_s)$ is an $E$-linear combination of
$e_j \otimes e^*_s$ $(j<s)$ and $e_s\otimes e^*_i$ ($i\geq t$). By
Lemma \ref{lem:Bst-proj} (\ref{it:Bst-proj-a}), $N\circ \pi_{s,s}$
is an $E$-linear combination of $\pi_{s,i}$ ($i\geq t$) and
$\pi_{j,s}$ ($j<s$). If the condition in (\ref{it:main-use-a})
holds, then $\tilde{\pi}_{s,s}([c])\in H^1(U_{\mu_s,0})$ is
contained in $ \ker(H^1(U_{\mu_s,0})\xrightarrow{N}
H^1(\BB_{\st,E}))$.

Similarly, $N(e_t\otimes e^*_t)$ is an $E$-linear combination of
$e_j\otimes e^*_t$ ($j\leq s$) and $e_t\otimes e^*_i$ ($i>t$). So
$N\circ \pi_{t,t}$ is an $E$-linear combination of $\pi_{j,t}$
($j\in J$) and $\pi_{t,i}$ ($i>t$). If the condition in
(\ref{it:main-use-b}) holds, then $\tilde{\pi}_{t,t}([c])\in
H^1(U_{\mu_t,0})$ is contained in
$\ker(H^1(U_{\mu_t,0})\xrightarrow{N} H^1(\BB_{\st,E}))$.

Now (\ref{it:main-use-a}) and (\ref{it:main-use-b}) follow from
Proposition \ref{prop:isom} and the fact that $H^1(E)$ is generated
by $\psi_1$ and $\psi_2$.

Next we prove (\ref{it:main-use-c}).

Let $\mathscr{D}_0$ be the filtered $E$-$(\varphi,N)$-submodule of
$\mathscr{D}$ generated by $e^*_t \otimes e_s$, $e^*_s\otimes e_s$
and $e^*_t\otimes e_t$, and let $\mathscr{D}_1$ be the filtered
$E$-$(\varphi,N)$-submodule of $\mathscr{D}$ generated by
$\mathscr{D}_0$ and $e^*_s\otimes e_t$. As $s$ is strongly marked
for $\CF$, there exist integers $k'_s$ and $k'_t$ satisfying
$k_s\leq k'_s< k'_t \leq k_t$ such that the filtration of the
filtered $E$-$(\varphi,N)$-submodule of $D/\CF_{s-1}D$ spanned by
$e_s$ and $e_t$ is given by
$$ \Fil^i =\left \{ \begin{array}{ll} E e_s \oplus E e_t & \text{ if }i\leq k_s, \\
  E (e_t +\CL_{\CF,s} e_s) & \text{ if }k_s<i\leq k'_t, \\
  0 & \text{ if } i> k'_t,  \end{array}\right. $$ and the filtration
of the filtered $E$-$(\varphi,N)$-submodule of $\overline{D}_2$
spanned by $e_s$ and $e_t$ is given by
$$ \Fil^i =\left \{ \begin{array}{ll} E e_s \oplus E e_t & \text{ if }i\leq k'_s, \\
  E (e_t +\CL_{\CF,s} e_s) & \text{ if }k'_s<i\leq k_t, \\
  0 & \text{ if } i> k_t.  \end{array}\right. $$ The dual
of the former coincides with the filtered
$E$-$(\varphi,N)$-submodule of $D^*_1$ spanned by $e^*_s$ and
$e^*_t$, with the filtration given by $$\Fil^i = \left\{
\begin{array}{ll} E e^*_s \oplus Ee^*_t & \text{ if } i \leq -k'_t \\
E (e^*_s-\CL_{\CF,s} e^*_t) & \text{ if } -k'_t<i\leq -k_s \\ 0 &
\text{ if } i>-k_s.
\end{array}\right.$$
Therefore,
$$ \Fil^0 (\mathscr{D}_1) = E e^*_t \otimes (e_t+\CL_{\CF,s} e_s)
\oplus E (e^*_s-\CL_{\CF,s} e^*_t)\otimes e_s \oplus E
(e^*_s-\CL_{\CF,s} e^*_t) \otimes ( e_t + \CL_{\CF,s} e_s )
$$ and
$$\Fil^0 (\mathscr{D}_0) = E (e^*_t \otimes e_t+\CL_{\CF,s} e^*_t\otimes e_s)
\oplus E (e^*_s\otimes e_s -\CL_{\CF,s} e^*_t\otimes e_s ). $$

We will construct a $1$-cocycle $c': G_{\BQ_p}\rightarrow
\BX_\st(\mathscr{D}_0)$ with $[c']\in
\ker(H^1(\BX_\st(\mathscr{D}_0))\rightarrow
H^1(\BX_\dR(\mathscr{D}_0)))$ such that
$\tilde{\pi}_{s,s}([c'])=\tilde{\pi}_{s,s}([c])$ and
$\tilde{\pi}_{t,t}([c'])=\tilde{\pi}_{t,t}([c])$.

As $c$ takes values in $\BX_\st(\mathscr{D})$, we have
$\varphi(c_\sigma)=c_\sigma$ and $N(c_\sigma)=0$. From
$\varphi(c_\sigma)=c_\sigma$ we obtain
$$ \varphi ( \pi_{j,i}(c_\sigma) ) = \alpha_i^{-1} \alpha_j  \pi_{j,i}(c_\sigma)
$$ for every $i\in I$ and every $j\in J$. In particular, we have
\begin{equation} \label{eq:varphi-value}
\varphi(\pi_{s,s}(c_\sigma))= \pi_{s,s}(c_\sigma), \
\varphi(\pi_{t,t}(c_\sigma)) = \pi_{t,t} (c_\sigma) , \
\varphi(\pi_{t,s}(c_\sigma)) = p \ \pi_{t,s} (c_\sigma).
\end{equation} By Lemma \ref{lem:Bst-proj} if $$ N(e_j\otimes e^*_i)= \sum_{(i',j')\in I\times
J}\lambda_{j',i'} e_{j'}\otimes e^*_{i'} ,$$ then
$$ N( \pi_{j,i}(c_\sigma) ) = \sum_{(i',j')\in I\times J} \lambda_{j',i'} \pi_{j',i'}(c_\sigma) .  $$
Since $N(e_t\otimes e^*_s)=e_s\otimes e^*_s - e_t\otimes e^*_t$, we
have
\begin{equation} \label{eq:mono-values}
N ( \pi_{t,s} (c_\sigma) ) = \pi_{s,s}(c_\sigma) - \pi_{t,t}
(c_\sigma).
\end{equation}

By Lemma \ref{lem:exact-Bcris-BdR} there exists some $y\in
\BB_{\st,E}^{\varphi=p}$ such that $N(y)= x_s-x_t$. As $\varphi$
commutes with $G_{\BQ_p}$, we have $\sigma(y)\in
\BB_{\st,E}^{\varphi=p}$ for every $\sigma\in G_{\BQ_p}$. Let $c'$
be the $1$-cocycle with values in $\BB_{\st,E}\otimes_E
\mathscr{D}_0$ defined by
\begin{eqnarray*}c': \sigma & \mapsto & (\pi_{t,s}(c_\sigma)- (\sigma-1) y ) e^*_t \otimes e_s  \\
&& + ( \pi_{s,s}(c_\sigma) - (\sigma-1)x_s ) e^*_s\otimes e_s + (
\pi_{t,t}(c_\sigma) - (\sigma-1) x_t ) e^*_t\otimes e_t .
\end{eqnarray*}

We show that $c'$ takes values in $\BX_\st(\mathscr{D}_0)$. What we
need to check is that $\varphi(c'_\sigma)=c'_\sigma$ and
$N(c'_\sigma)=0$. By (\ref{it:main-use-a}), (\ref{it:main-use-b})
and the definition of $c'_\sigma$ we have
\begin{equation} \label{eq:pi-a}
\pi_{s,s}(c'_\sigma), \pi_{t,t}(c'_\sigma)\in E \subset
\BB_{\st,E}^{\varphi=1, N=0}.\end{equation}  By
(\ref{eq:varphi-value}), $\pi_{t,s}(c_\sigma)$ is in
$\BB_{\st,E}^{\varphi=p}$. From this and the fact $(\sigma-1)y\in
\BB_{\st,E}^{\varphi=p}$ we get
\begin{equation} \label{eq:pi-b} \pi_{t,s}(c'_\sigma) \in
\BB_{\st,E}^{\varphi=p}.
\end{equation}
From (\ref{eq:mono-values}) and the fact $N(y)=x_s-x_t$ we obtain
\begin{equation} \label{eq:pi-c} N (\pi_{t,s}(c'_\sigma)) =
\pi_{s,s}(c'_\sigma)-\pi_{t,t}(c'_\sigma) . \end{equation}
Equalities (\ref{eq:pi-a}), (\ref{eq:pi-b}) and (\ref{eq:pi-c})
ensure that $\varphi(c'_\sigma)=c'_\sigma $ and $N(c'_\sigma)=0$ for
every $\sigma\in G_{\BQ_p}$.

From the definition of $c'$ we see that
$\tilde{\pi}_{s,s}([c'])=\tilde{\pi}_{s,s}([c])$ and
$\tilde{\pi}_{t,t}([c'])=\tilde{\pi}_{t,t}([c])$. By Lemma
\ref{lem:aux} to finish our proof we only need to show that the
image of $[c']$ in $H^1(\BX_\dR(\mathscr{D}_0))$ is zero.

By Lemma \ref{lem:basis} there exist $a_{i_1, i_2} \in E$
($i_1,i_2\in I$, $i_1>i_2$)  such that $f_i:=e_i+\sum_{i'\in I,
i'<i}a_{i, i'}e_{i'}$ ($i\in I$) form an $E$-basis of
$\overline{D}_2$ compatible with the filtration on $\overline{D}_2$.
Similarly, there exist $b_{j_1, j_2} \in E$ ($j_1,j_2\in J$,
$j_1<j_2$) such that $g_j:=e^*_j+\sum_{j'\in J, j'>j}b_{j,
j'}e^*_{j'}$ ($j\in J$) form an $E$-basis of $D^*_1$ compatible with
the filtration. Then $\{g_j\otimes f_i:i\in I,j\in J\}$ is an
$E$-basis of $\mathscr{D}$ compatible with the filtration. Note that
$a_{t,s}=-b_{s,t}=\CL_{\CF,s}$ and $$ f_s=e_s, \ \ f_t= e_t
+\CL_{\CF,s}e_s, \ \ g_t=e^*_t, \ \ g_s=e^*_s-\CL_{\CF,s}e^*_t. $$
As a consequence, $\{g_t\otimes f_s, g_s\otimes f_s, g_t\otimes
f_t\}$ is an $E$-basis of $\mathscr{D}_0$ compatible with the
filtration.

Conversely, there are $\tilde{a}_{i_1, i_2}$ ($i_1,i_2\in I$,
$i_1>i_2$) and $\tilde{b}_{j_1, j_2} $ ($j_1,j_2\in J$, $j_1<j_2$)
in $E$ such that $e_i=f_i+\sum_{i'\in I, i'<i}\tilde{a}_{i,
i'}f_{i'}$ and $e^*_j=g_j+\sum_{j'\in J, j'>j}\tilde{b}_{j,
j'}g_{j'}$. Note that
$-\tilde{a}_{t,s}=\tilde{b}_{s,t}=\CL_{\CF,s}$.

Expressing $c$ in terms of the basis $\{g_j\otimes f_i: i\in I, j\in
J\}$ we obtain
$$ c = \sum_{i'\in I, j'\in J} (\pi_{j',i'}(c)+ \sum_{i>i'}\tilde{a}_{i,i'}\pi_{j',i}(c)  +
\sum_{j<j'}\tilde{b}_{j,j'}\pi_{j,i'}(c)+\sum_{i>i',
j<j'}\tilde{a}_{i,i'}\tilde{b}_{j,j'}\pi_{j,i}(c)) g_{j'}\otimes
f_{i'}.
$$ In particular, the coefficient of $g_t\otimes f_s$ is
\begin{equation}\label{eq:coef-f-g}
\pi_{t,s}(c)+ \sum_{i\geq t}\tilde{a}_{i,s}\pi_{t,i}(c)  +
\sum_{j\leq s}\tilde{b}_{j,t}\pi_{j,s}(c)+\sum_{i\geq t, j\leq
s}\tilde{a}_{i,s}\tilde{b}_{j,t}\pi_{j,i}(c) . \end{equation} As the
image of $[c]$ in
$H^1(\BB_{\dR,E}\otimes_E\mathscr{D}/\Fil^0(\BB_{\dR,E}\otimes_E\mathscr{D}))$
is zero, the image of the $1$-cocycle (\ref{eq:coef-f-g}) in
$H^1(\BB_{\dR,E}/ \Fil^{k'_t-k'_s} \BB_{\dR,E})$ is zero. As the
images of $\pi_{t,i}(c)$ $(i>t)$, $\pi_{j,s}(c)$ ($j<s$) and
$\pi_{j,i}(c)$ ($i\geq t, j\leq s$) in $H^1(\BB_{\st,E})$ are zero,
their images in $H^1(\BB_{\dR,E}/ \Fil^{k'_t-k'_s} \BB_{\dR,E})$ are
also zero. This implies that the image of the $1$-cocycle
$$ \pi_{t,s}(c) + \tilde{a}_{t,s}\pi_{t,t}(c) +\tilde{b}_{s,t} \pi_{s,s}(c) $$ in $H^1(\BB_{\dR,E}/
\Fil^{k'_t-k'_s} \BB_{\dR,E})$ is zero, and so is the image of the
$1$-cocycle
$$ \pi_{t,s}(c') + \tilde{a}_{t,s}\pi_{t,t}(c') +\tilde{b}_{s,t}
\pi_{s,s}(c'). $$

Now
$$ c' = (\pi_{t,s}(c') + \tilde{a}_{t,s}\pi_{t,t}(c') +
\tilde{b}_{s,t} \pi_{s,s}(c')) g_t\otimes f_s + \pi_{s,s}(c')
g_s\otimes f_s +  \pi_{t,t}(c') g_t\otimes f_t . $$ Since
$g_s\otimes f_s, g_t\otimes f_t \in \Fil^0 \mathscr{D}_0$, the image
of $[c']$ in
$H^1(\BB_{\dR,E}\otimes_E\mathscr{D}_0/\Fil^0(\BB_{\dR,E}\otimes_E\mathscr{D}_0))$
is zero if and only if the image of the $1$-cocycle $\pi_{t,s}(c') +
\tilde{a}_{t,s}\pi_{t,t}(c') +\tilde{b}_{s,t} \pi_{s,s}(c')$ in
$H^1(\BB_{\dR,E}/ \Fil^{k'_t-k'_s} \BB_{\dR,E})$ is zero, which is
observed above.
\end{proof}

Now let $V$ be a semistable $E$-representation of $G_{\BQ_p}$, $D$
the associated filtered $E$-$(\varphi,N)$-module. Suppose that
$\varphi$ is semisimple on $D$ and let $\CF$ be a refinement on $D$.
Assume that $s\in \{1, \dots, n-1\}$ is strongly marked for $\CF$,
and $t=t_\CF(s)$. Let $\{e_1, \dots, e_n\}$ be an $s$-perfect basis
for $\CF$.

The composition of $V^*\otimes_E V \rightarrow \BX_\st(D^*\otimes_E
D) $ and
$$ \pi_{j,i}: \BB_{\st,E}\otimes_E \mathscr{D} \rightarrow \BB_{\st,E} , \ \ \sum_{h=1}^n\sum_{\ell=1}^n b_{h,\ell} e^*_{h}\otimes e_{\ell} \mapsto
b_{j,i},
$$
is again denoted by $\pi_{j,i}$, which is $G_{\BQ_p}$-equivariant.

\begin{cor}\label{cor:main-use}
Let $c:G_{\BQ_p}\rightarrow V^*\otimes_E V$ be a $1$-cocycle. If
$\pi_{j,i} ([c])=0$ when $j<i$, then there are $x_s,x_t\in
\BB_{\st,E}^{\varphi=1}$,
 and $\gamma_{s,1},
\gamma_{s,2}, \gamma_{t,1},\gamma_{t,2}\in E$ such
that$$\pi_{s,s}(c_\sigma)=\gamma_{ s,1}\psi_1(\sigma)+\gamma_{s,2}
\psi_2(\sigma)+(\sigma-1)x_s$$ and
$$\pi_{t,t}(c_\sigma)=\gamma_{t,1} \psi_1(\sigma)+\gamma_{t,2}
\psi_2(\sigma)+(\sigma-1)x_t .$$
 Furthermore $
\gamma_{s,1}-\gamma_{t,1} = \CL_{\CF,s} ( \gamma_{s,2}-\gamma_{t,2}
) $.
\end{cor}
\begin{proof} We form the quotient $\mathscr{D}$ of
$D^*\otimes_ED$ as at the beginning of this subsection. Then we have
the following commutative diagram
\[ \xymatrix{
H^1(V^*\otimes_E V) \ar[r] & H^1(\BX_\st(D^*\otimes_E D)) \ar[r]\ar[d] & H^1(\BX_\dR(D^*\otimes_E D ))  \ar[d] \\
& H^1(\BX_\st(\mathscr{D})) \ar[r] & H^1(\BX_\dR(\mathscr{D})) }
\] where the upper horizontal line is exact, which implies that the image of $[c]$ in
$H^1(\BX_\st(\mathscr{D}))$ belongs to
$\ker(H^1(\BX_\st(\mathscr{D}))\rightarrow
H^1(\BX_\dR(\mathscr{D})))$. Hence the assertion follows from
Theorem \ref{prop:main-use}.
\end{proof}

\section{Projection vanishing property} \label{sec:aux-res}

We will attach to every infinitesimal deformation of a
representation of $G_{\BQ_p}$ i.e. an $S$-representation of
$G_{\BQ_p}$ a $1$-cocycle, and show that, when the
$S$-representation admits a triangulation and the residue
representation is semistable, the corresponding $1$-cocycle has the
projection vanishing property. Here, $S=E[Z]/(Z^2)$.

Let $\CV$ be an $S$-representation of $G_{\BQ_p}$,
$\CM=\BD_\rig(\CV)$. Suppose that $\CM$ admits a triangulation
$\Fil_\bullet$. Let $(\delta_{1}, \dots, \delta_{n})$ be the
corresponding triangulation data.

Let $z$ be the closed point defined by the maximal ideal $(Z)$ of
$S$. Suppose that $\CV_z$, the evaluation of $\CV$ at $z$, is
semistable, and let $D_z$ be the filtered $E$-$(\varphi,N)$-module
attached to $\CV_z$. Suppose that $\varphi$ is semisimple on $D_z$.
Let $\CF$ be the refinement of $D_z$ corresponding to the induced
triangulation of $\CM_z$. Let $\{ e_{1,z},\dots, e_{n,z} \}$ be an
ordered basis of $D_z$ perfect for $\CF$. Write
$\varphi(e_{i,z})=\alpha_{i,z}e_{i,z}$.

For $i=1,\dots, n$ there exists a continuous additive character
$\epsilon_i$ of $\BQ_p^\times$ with values in $E$ such that
$\delta_i = \delta_{i,z} (1+Z \epsilon_i)$.  By identifying $\Gamma$
with $\BZ_p^\times$ via $\chi_{\mathrm{cyc}}$ we consider
$\epsilon_i|_{\BZ_p^\times}$ as a character of $\Gamma$ or a
character of $G_{\BQ_p}$ that factors through $\Gamma$, denoted by
$\epsilon'_i$.

Fix an $S$-basis $\{v_1,\dots, v_n\}$ of $\CV$, and write the matrix
of $\sigma\in G_{\BQ_p}$ for this basis, $B_\sigma$, in the form
\begin{equation} \label{eq:AB-sigma}
B_\sigma = (I_n + Z U_\sigma ) A_\sigma
\end{equation}
with $A_\sigma\in \GL_n(E)$ and $U_\sigma\in \mathrm{M}_n(E)$. Then
$\{v_{1,z},\dots, v_{n,z}\}$ is an $E$-basis of $\CV_z$, and
$A_\sigma$ is the matrix of $\sigma$ for this basis. For every
$\sigma\in G_{\BQ_p}$ put $$ c_\sigma = \sum_{i,j} (U_{\sigma})_{ij}
v_{j,z}^* \otimes  v_{i,z} . $$

\begin{lem}
$\sigma\mapsto c_\sigma$ is a $1$-cocycle of $G_{\BQ_p}$ with values
in $ \CV_z^*\otimes_E \CV_z $.
\end{lem}
\begin{proof} From (\ref{eq:AB-sigma}) we obtain $U_{\sigma \tau} = U_\sigma+ A_\sigma U_\tau
A_\sigma^{-1}$. In other words, for every $i,j\in \{1,\dots, n\}$,
$$ (U_{\sigma\tau})_{ij} = (U_\sigma)_{ ij} + \sum_{h,\ell} (A_\sigma)_{ ih} (U_\tau)_{ h\ell} (A^{-1}_\sigma)_{\ell j}.
$$ Hence {\allowdisplaybreaks
$$\begin{aligned} c_{\sigma\tau} & =  \sum_{i,j}  (U_{\sigma\tau})_{
ij}
v_{j,z}^*\otimes v_{i,z} \\
&= \sum _{i,j} \Big( (U_\sigma)_{ij} + \sum_{h,\ell} (A_\sigma)_{ih}
(U_\tau)_{h\ell} (A^{-1}_\sigma)_{\ell j} \Big) v_{j,z}^* \otimes v_{i,z} \\
&= \sum_{i,j} (U_\sigma)_{ij} v_{j,z}^*\otimes v_{i,z} +
\sum_{h\ell} (U_\tau)_{h\ell}
\Big( \sum_j ( A^{-1}_\sigma )_{\ell j} v_{j,z}^* \Big)\otimes\Big(\sum_i(A_\sigma)_{ih}v_{i,z}\Big) \\
&= c_\sigma + \sum_{h\ell} (U_\tau)_{h\ell}
(v_{\ell,z}^*)^\sigma \otimes (v_{h,z})^\sigma  \\
&= c_{\sigma} + c_\tau^\sigma,
\end{aligned}$$}

\noindent as desired.
\end{proof}

Let $x_{ij}\in \BB_{\st,E}$ ($i,j=1,\dots, n$) be such that
\begin{equation}\label{eq:relation-e-and-v}
e_{j,z} = x_{1j} v_{1,z} +\cdots + x_{nj} v_{n,z} .   \end{equation}
Then $X=(x_{ij})$ is in $\GL_n(\BB_{\st,E})$. As $e_{1, z},\dots,
e_{n,z}$ are fixed by $G_{\BQ_p}$, we have $X^{-1}A_\sigma
\sigma(X)=I_n$ for all $\sigma\in G_{\BQ_p}$. For $j=1,\dots, n$ put
$e_j = x_{1j} v_1 +\cdots + x_{nj}  v_n$. Then $\{e_1,\dots, e_n\}$
is a basis of $\BB_{\st,E}\otimes_E \CV$ over $\BB_{\st,E}\otimes_E
S$. (Note that $\BB_{\st,E}\widehat{\otimes}_E
S=\BB_{\st,E}\otimes_E S$ and $\BB_{\st,E}\widehat{\otimes}_E
\CV=\BB_{\st,E}\otimes_E \CV$.)

\begin{lem} \label{lem:trivial-useful} For $i=1,\dots, n$ we have
$ \varphi(e_i)=\alpha_{i,z}e_i.$
\end{lem}
\begin{proof} As $v_{1,z},\dots, v_{n,z}$ are fixed by $\varphi$, from
$\varphi(e_{j,z})=\alpha_{j,z}e_{j,z}$ ($i=1,\dots, n$) we obtain
$\varphi(x_{ij})=\alpha_{j,z}x_{ij}$ for every $j$. Thus
$\varphi(e_j)=\sum\limits_i \varphi(x_{ij})v_i=
\sum\limits_i\alpha_{j,z}x_{ij}v_i=\alpha_{j,z}e_j$.
\end{proof}

The matrix of $\sigma$ for the basis $\{e_1,\dots, e_n\}$ is
$$ X^{-1}B_\sigma \sigma(X) = I_n + Z  X^{-1}U_\sigma X . $$
A simple computation shows that
$$ c_\sigma =\sum_{i,j} (X^{-1}U_\sigma X)_{ij} e^*_{j,z} \otimes
e_{i,z}.
$$ Let $\pi_{h \ell}$ be the projection
\begin{equation}\label{eq:proj}
\BB_{\st,E}\otimes_E(\CV_z \otimes_E \CV_z^*)\rightarrow \BB_{\st,E}
, \ \ \ \sum_{j,i} b_{ij}e^*_{j,z}\otimes e_{i,z} \mapsto b_{h \ell}
.
\end{equation}

\begin{lem}\label{lem:middle-step}
Let $\delta'_i$ be the character $1+Z\epsilon'_i$. Then for
$h=1,\dots, n$ there is an element in $$[\BB_{\cris,
E}^{\varphi=\prod_{i=1}^h (\alpha_{i,z} (1+
Z\epsilon_i(p)))}\otimes_E(\wedge^h
\CV)({\delta'_1}^{-1}\cdots{\delta'_h}^{-1})]^{G_{\BQ_p}}$$ denoted
by $g_{1,\dots, h}$, whose image in $\BB_{\st,E}\otimes_E
\wedge^{h}\CV_z$ is exactly $ e_{1,z}\wedge \cdots \wedge e_{h,z}$.
\end{lem}
\begin{proof} Put $f_i=w_{\delta_i}$.  By Proposition
\ref{prop:triang-refine} we have $\alpha_{i, z}=\delta_{i,z}(p)
p^{f_{i,z}}$ and $\delta_{i,z}(x) = x^{f_{i,z}}$ for every $x\in
\BZ_p^\times$.

For $i=1,\dots, n$ let $m_i$ be a nonzero element in $\Fil_i\CM$
such that
$$\varphi(m_i)\equiv \delta_i(p)m_i \mod \Fil_{i-1}\CM$$ and $$
\gamma(m_i)\equiv \delta_i(\chi_\cyc(\gamma))m_i \mod
\Fil_{i-1}\CM$$ for every $\gamma\in \Gamma.$ Then $m_1\wedge \cdots
\wedge m_h$ is a nonzero element in $$ (\wedge^h
\CM)^{\varphi=(\delta_1\cdots\delta_h)(p),
\Gamma=(\delta_1\cdots\delta_h)|_{\BZ_p^\times}} .$$ Considered as
an element in $(\wedge^h
\CM)({\delta'_1}^{-1}\cdots{\delta'_h}^{-1})[\frac{1}{t_\cyc}]$,
$t_{\cyc}^{f_{1,z}+\cdots + f_{h,z}}m_1\wedge \cdots \wedge m_h$  is
in
\begin{eqnarray*} && [(\wedge^h
\CM)({\delta'_1}^{-1}\cdots{\delta'_h}^{-1})[\frac{1}{t_\cyc}]]^{\varphi
= \prod_{i=1}^h (\alpha_{i,z} (1+ Z \epsilon_i(p))),\Gamma=1} \\ & =
& \BD_{\cris}((\wedge^h
\CV)({\delta'_1}^{-1}\cdots{\delta'_h}^{-1}))^{\varphi=\prod_{i=1}^h
(\alpha_{i,z} (1+ Z\epsilon_i(p)))} \\
&=& [\BB_{\cris, E}^{\varphi=\prod_{i=1}^h (\alpha_{i,z} (1+
Z\epsilon_i(p)))}\otimes_E(\wedge^h
\CV)({\delta'_1}^{-1}\cdots{\delta'_h}^{-1})]^{G_{\BQ_p}},
\end{eqnarray*} where the first equality follows from \cite[Proposition 3.7]{Be} and the second is
obvious.

Let $\BB_{\log,\BQ_p}^\dagger$ be the ring used in \cite{Be}. As the
refinement corresponding to $\Fil_{\bullet,z}$ is $\CF$, we have
$$ [(\BB_{\log, \BQ_p}^\dagger \otimes_{\BQ_p} E)[\frac{1}{t_\cyc}]\otimes_{\CR_E} ( \Fil_{i,z}\CM_z)]^{\Gamma=1} =
\CF_i
D_z.
$$ Since the image of $t_\cyc^{f_{i,z}}m_{i,z}$ in
$\CR_E(\delta_i)[\frac{1}{t_\cyc}]$ is fixed by $\Gamma$, we have
$$ e_{i,z} \equiv t_\cyc^{f_{i,z}}m_{i,z} \mod (\BB_{\log, \BQ_p}^\dagger \otimes_{\BQ_p} E)[\frac{1}{t_\cyc}]\otimes_{\CR_E} \Fil_{i-1,z}\CM_z
$$ up to a nonzero scalar. This implies that $t_{\cyc}^{f_{1,z}+\cdots + f_{h,z}}m_{1}\wedge \cdots \wedge m_{h} \ \mathrm{mod} \ Z$ coincides
with $e_{1,z}\wedge \cdots \wedge e_{h,z}$ up to a nonzero scalar.
\end{proof}

\begin{thm} \label{prop:middle-step}
\begin{enumerate}
\item\label{it:middle-step-a} For every pair of integers $(h,\ell)$ such that $h<\ell$ we have
$\pi_{h \ell}([c])=0$.
\item\label{it:middle-step-b} For every $h=1,\dots, n$,
$\pi_{h,h}([c])$ coincides with the image of $[\epsilon'_h]$ in
$H^1(\BB_{\st,E})$.
\end{enumerate}
\end{thm} We consider (\ref{it:middle-step-a}) as the projection
vanishing property.
\begin{proof}
Let $g_{1,\dots,h}$ be as in Lemma \ref{lem:middle-step}. Write
\begin{equation} \label{eq:exp-f} g_{1,\dots,h}=e_{1}\wedge
\cdots\wedge e_{h} + Z \sum_{J} \lambda_{J} e_{J}, \end{equation}
where $\lambda_J\in \BB_{\st,E}$ and $J$ runs over all subsets of
$\{1,\dots, n\}$ with cardinal number $h$. Here, if $J=\{j_1 <
\cdots < j_{h}\}$, then $e_J=e_{j_1}\wedge \cdots \wedge e_{j_{h}}$.

As the matrix of $\sigma\in G_{\BQ_p}$ for the basis $\{e_1,\dots,
e_n\}$ is $I_n+ Z X^{-1}U_\sigma X$, we have
$$ \sigma(e_i) = e_i + \sum_{j=1}^n Z (X^{-1}U_\sigma X)_{ji}e_j .$$
Hence
\[\begin{aligned}
g_{1,\dots,h}=\sigma(g_{1,\dots,h}) & = [1 - Z
\epsilon'_{1}(\sigma)- \cdots - Z
\epsilon'_h(\sigma)] \\
& \hskip 15pt \times \Big[  \Big(e_{1} + Z
\sum_{j=1}^n (X^{-1}U_\sigma X)_{j1} e_j\Big)\wedge\cdots  \\
& \hskip 60pt \wedge\Big( e_h + Z \sum_{j=1}^n (X^{-1}U_\sigma
X)_{jh}e_j \Big)   + Z \sum_{J} \sigma(\lambda_{J}) e_{J} \Big].
\end{aligned} \]
For every $\ell=h+1, \dots,n$, comparing the coefficient of $e_1
\wedge \cdots \wedge e_{h-1} \wedge e_\ell$ in the right hand side
of the above equality and the right hand side of (\ref{eq:exp-f}),
we obtain $$ \lambda_{1, \dots, h-1, \ell} =
\sigma(\lambda_{1,\dots, h-1,\ell}) + (X^{-1}U_\sigma X)_{\ell h},
$$ which proves (\ref{it:middle-step-a}).

Similarly, comparing the coefficients of $e_1\wedge \cdots \wedge
e_h$ in the above two expressions for $g_{1,\dots,h}$ we obtain $$
\lambda_{1,\dots,h} = \sigma(\lambda_{1,\dots,h}) + \sum_{i=1}^h
(X^{-1}U_\sigma X)_{ii} -\sum_{i=1}^h \epsilon'_i(\sigma). $$ Thus
we have
\begin{equation}\label{eq:lambda} (X^{-1}U_\sigma X)_{hh}- \epsilon'_h(\sigma)=
(\sigma-1)(\lambda_{1,\dots, h-1}-\lambda_{1,\dots, h}),
\end{equation}
which implies (\ref{it:middle-step-b}).
\end{proof}

\begin{cor} \label{cor:middle-step} For $h=1,\dots, n$, there exist $\gamma_{h,1},\gamma_{h,2}\in E$ and $\xi_h\in
\BB_{\st,E}^{\varphi=1}$ such that for every $\sigma\in G_{\BQ_p}$,
$$ (X^{-1} U_\sigma X)_{hh} = \gamma_{h,1} \psi_1(\sigma) +\gamma_{h,2} \psi_2(\sigma) +
(\sigma-1)\xi_h.
$$
\end{cor}
\begin{proof} By Theorem \ref{prop:middle-step}
(\ref{it:main-use-a}), $\pi_{j,h}([c])$ ($j<h$) and $\pi_{h,i}([c])$
($i>h$) are zero. Repeating the argument in the proof of Theorem
\ref{prop:main-use} (\ref{it:main-use-a}) (\ref{it:main-use-b})
yields our assertion.
\end{proof}

As $\psi_2$ is an $E$-basis of $\mathrm{Hom}_{\mathrm{cont}}(\Gamma,
E)$, the $E$-vector space of continuous homomorphisms from $\Gamma$
to $E$, there exists $\epsilon_{h,2}\in E$ such that
$\epsilon'_h=\epsilon_{h,2}\psi_2$.

\begin{lem} \label{lem:constant-gamma} We have
$\gamma_{h,1}=-\epsilon_h(p)$ and $\gamma_{h,2}=\epsilon_{h,2}$.
\end{lem}
\begin{proof} We keep to use notations in the proof of Theorem
\ref{prop:middle-step}. By (\ref{eq:lambda}) and Corollary
\ref{cor:middle-step} we have
\begin{eqnarray*} (\sigma-1)(\lambda_{1,\dots, h}-\lambda_{1,\dots, h-1} ) & = &
\epsilon_{h,2}\psi_2(\sigma)-(X^{-1} U_\sigma X)_{hh} \\
&=&
-\gamma_{h,1}\psi_1(\sigma)+(\epsilon_{h,2}-\gamma_{h,2})\psi_2(\sigma)-(\sigma-1)\xi_h,
\end{eqnarray*} with the convention that $\lambda_{1,\dots, h-1}=0$
when $h=1$. Note that there exists $\omega\in
\mathrm{W}(\overline{\BF}_p)$ such that $\varphi(\omega)-\omega=1$,
where $\mathrm{W}(\overline{\BF}_p)$ is the ring of Witt vectors
with coefficients in the algebraic closure of $\BF_p$. Then
$(\sigma-1)\omega=\psi_1(\sigma)$. Hence
$$ (\epsilon_{h,2}-\gamma_{h,2})\psi_2(\sigma) = (\sigma-1)(\lambda_{1,\dots,h}-\lambda_{1,\dots, h-1}+\xi_h +\gamma_{h,1}\omega) .
$$ As the extension of $\BQ_p$ by $\BQ_p$ corresponding to $\psi_2$
is not Hodge--Tate, we have $\gamma_{h,2} = \epsilon_{h,2}$ and
$\lambda_{1,\dots,h}-\lambda_{1,\dots, h-1}+\xi_h
+\gamma_{h,1}\omega \in E$. Then \begin{equation}
\label{eq:lambda-a} (\varphi-1)(\lambda_{1,\dots,
h}-\lambda_{1,\dots, h-1}) = -(\varphi-1)\xi_h - \gamma_{h,1}
(\varphi-1)\omega = -\gamma_{h,1} .
\end{equation}

Note that $ \oplus_{I} Z e_I $, where $I$ runs over subsets of
$\{1,\dots,n\}$ with cardinal number $h$ except $\{1,\dots, h\}$, is
stable by $\varphi$. Let $Y$ denote this subspace. Then we have
$$ \varphi(g_{1,\dots, h}) =  ( 1+Z \varphi(\lambda_{1,\dots,h}) )(\prod_{i=1}^h\alpha_{i,z}) e_1\wedge\cdots \wedge e_h  \hskip 10pt (\text{mod } Y).
$$ On the other hand,
\begin{eqnarray*} \varphi(g_{1,\dots,h}) & = & (1+Z\sum_{i=1}^h \epsilon_i(p)) (
\prod_{i=1}^h \alpha_{i,z}) g_{1,\dots,h} \\
& = & (1+Z\sum_{i=1}^h \epsilon_i(p)) ( \prod_{i=1}^h \alpha_{i,z})
(1+ Z \lambda_{1,\dots,h}) e_1\wedge\cdots\wedge e_h \hskip 10pt
(\text{mod } Y).  \end{eqnarray*} Hence we obtain
\begin{equation} \label{eq:lambda-h}
(\varphi-1)\lambda_{1,\dots, h} = \sum\limits_{i=1}^h
\epsilon_i(p).\end{equation} By (\ref{eq:lambda-a}) and
(\ref{eq:lambda-h}) we have
$$ \gamma_{h,1} = -(\varphi-1) (\lambda_{1,\dots,h}-\lambda_{1,\dots,h-1}) = -\epsilon_h(p), $$ as
wanted.
\end{proof}

\section{Proof of the main theorem}\label{sec:proof-main}

Let $S$ be an affinoid algebra over $E$. Let $\CV$ be a trianguline
$S$-representation of $G_{\BQ_p}$, $\CM=\BD_\rig(\CV)$. Fix a
triangulation of $\CM$ and let $(\delta_1, \dots, \delta_n)$ be the
corresponding triangulation data.

We restate our main theorem as follows.

\begin{thm}\label{thm:main-use} Let $z$ be a closed point of $S$
such that $\CV_z$ is semistable. Let $D_z$ be the filtered
$E$-$(\varphi,N)$-module attached to $\CV_z$, and suppose that
$\varphi$ is semisimple on $D_z$. Let $\CF$ be the refinement of
$D_z$ corresponding to the triangulation of $\CM_z$. If $s$ is
strongly marked for $\CF$ and $t=t_\CF(s)$, then
$$ \frac{\mathrm{d}\delta_t(p)}{\delta_t(p)}-\frac{\mathrm{d}\delta_s(p)}{\delta_s(p)}
+ \CL_{\CF,s}(\mathrm{d}w_{\delta_t}-\mathrm{d}w_{\delta_s})
$$ is zero at $z$. Here, $\CL_{\CF,s}$ is the invariant defined in
Definition \ref{defn:Fontaine-Mazur}.
\end{thm}

Since we only need the first order derivation, we may assume that
$S=E[Z]/Z^2$ and $z$ corresponds to the maximal ideal $(Z)$.

For $i=1,\dots, n$ there exists a continuous additive character
$\epsilon_i$ of $\BQ_p^\times$ with values in $E$ such that
$\delta_i = \delta_{i,z} (1+Z \epsilon_i)$.  By identifying $\Gamma$
with $\BZ_p^\times$ via $\chi_{\mathrm{cyc}}$ we consider
$\epsilon_i|_{\BZ_p^\times}$ as a character of $G_{\BQ_p}$ that
factors through $\Gamma$. Then there exists $\epsilon_{i,2}\in E$
such that $\epsilon_i|_{\BZ_p^\times}= \epsilon_{i,2}\psi_2$.
Clearly $w_{\delta_i}= w_{\delta_{i,z}}+ Z \epsilon_{i,2}$. Thus
$$ \frac{\mathrm{d}\delta_i(p)}{\delta_i(p)} = \epsilon_i(p)\mathrm{d}Z, \hskip 10pt \mathrm{d}w_{\delta_i} = \epsilon_{i,2} \mathrm{d}Z . $$
Hence Theorem \ref{thm:main-use} comes from the following

\begin{prop}\label{prop:a-mu} $ \epsilon_t(p)-\epsilon_s(p) + \CL_{\CF,s} (\epsilon_{t,2} -\epsilon_{s,2})=0$.
\end{prop}
\begin{proof} Let $c$ be the $1$-cocycle attached to the
infinitesimal deformation $\CV$ of $\CV_z$. Fix an $s$-perfect basis
for $\CF$,  
and let $\pi_{h \ell}$ ($h,\ell \in \{ 1 , \dots , n \}$) be the
maps defined by (\ref{eq:proj}) using this basis. By Corollary
\ref{cor:middle-step} and Lemma \ref{lem:constant-gamma} there are
$\xi_s,\xi_t\in \BB_{\st,E}^{\varphi=1}$ such
that$$\pi_{s,s}(c_\sigma)=-\epsilon_{s}(p)\psi_1(\sigma)+
\epsilon_{s,2} \psi_2(\sigma)+(\sigma-1)\xi_s$$ and
$$\pi_{t,t}(c_\sigma)=-\epsilon_{t}(p) \psi_1(\sigma)+\epsilon_{t,2}
\psi_2(\sigma)+(\sigma-1)\xi_t .$$  By Theorem
\ref{prop:middle-step} (\ref{it:middle-step-a}) we have
$\pi_{h\ell}([c])=0$ when $h<\ell$. Thus it follows from Corollary
\ref{cor:main-use} that $ \epsilon_{t}(p)-\epsilon_{s}(p) =
\CL_{\CF,s} ( \epsilon_{s,2}-\epsilon_{t,2} ) $.
\end{proof}

\section{Relation with the $p$-adic Langlands
program}\label{sec:apply}

We discuss a possible application of Theorem \ref{thm:main} in
$p$-adic Langlands program.

We work on the simple Shimura varieties used in \cite{HT} and
\cite{TY2007}. Let $F$ be an imaginary quadratic field, $B$ a
division algebra with the center $F$ and $\dim_F B=n^2$. Suppose
that $F$ and $B$ satisfy the conditions in \cite[\S I.7]{HT}. Pick a
positive involution $*$ on $B$ such that $*|_F$ is the complex
conjugacy. For some fixed nonzero $\beta\in B^{*=-1}$ we define an
involution on $B$ by $x^\#= \beta x^* \beta^{-1}$, and a reductive
group $\mathbf{G}/\BQ$ such that for any $\BQ$-algebra $R$,
$\mathbf{G}(R)=\{(\lambda, g)\in R^\times \times
(B^{\mathrm{op}}\otimes_\BQ R)^\times: gg^\#=\lambda\}$. Write
$A_f=\widehat{\BZ}\otimes_\BZ\BQ$.  For each (sufficiently small)
compact open subgroup $K$ of $\mathbf{G}(A_f)$, let $\Sh_K$ be the
variety representing the moduli functor given in \cite[\S III.1]{HT}
or \cite[\S 2]{TY2007}. Then $\Sh_K$ is an $(n-1)$-dimensional
projective variety over $F$. The projective system $\{\Sh_K\}_K$ for
varying $K$ has a natural action of $\mathbf{G}(A_f)$.

Let $p$ be a prime number that splits in $F$ and $B$. Then
$\mathbf{G}(\BQ_p)$ is isomorphic to
$\BQ_p^\times\times\mathrm{GL}_n(\BQ_p)$, and for each prime
$\mathfrak{p}$ of $F$ above $p$, $F_\mathfrak{p}$ is isomorphic to
$\BQ_p$.

Let $E$ be a finite extension of $\BQ_p$, $\mathcal{O}_E$ the ring
of integers in $E$, $\varpi_E$ a uniformizer of $\mathcal{O}_E$. Let
$\rho$ be a self-dual $n$-dimensional continuous representation of
$G_F$ over $E$ appeared in the \'etale cohomology
$H^{n-1}_{\mathrm{et}}$ of local systems on $\Sh_K$ for $K$
sufficiently small \cite{HT, TY2007}. Put
$$ \widetilde{H}^{n-1}(K^p, E) : = \lim_{\overleftarrow{\;\;\:r\;\;}} \lim_{\overrightarrow{\;K'_p\;}}
H^{n-1}_{\mathrm{et}} ((\Sh_{K^pK'_p})_{\overline{\BQ}},
\mathcal{O}_E/\varpi_E^{r}\mathcal{O}_E) \otimes_{\mathcal{O}_E} E
$$ where $K^p$ denotes the component of $K$ outside $p$, and $K'_p$
runs over compact open subgroups of $\mathbf{G}(A_f)$. This is an
$E$-Banach space equipped with a continuous action of $\BQ_p^\times
\times\mathrm{GL}_n(\BQ_p)\times G_{F} \times \mathcal{H}^p$, where
$\mathcal{H}^p$ denotes the $E$-algebra of Hecke operators outside
$p$. Fix an $E$-valued smooth character $\psi_0$ of $\BQ_p^\times$.
Put
$$ \widetilde{\Pi}(\rho)^{\psi_0} := \mathrm{Hom}_{G_F} (\rho, \widetilde{H}^{n-1}(K^p,
E)^{\BQ_p^\times=\psi_0}).
$$ The representation of $\mathrm{GL}_n(\BQ_p)$, $\widetilde{\Pi}(\rho)^{\psi_0}$, is supposed to be the direct sum of several copies of
the admissible Banach representation
$\widehat{\Pi}(\mathrm{rec}(\psi_0)\otimes\rho)$ of
$\mathrm{GL}_n(\BQ_p)$ corresponding to $\mathrm{rec}(\psi_0)
\otimes\rho_\mathfrak{p}$ in the $p$-adic Langlands correspondence
(with certain normalization), where $\mathrm{rec}(\psi_0)$ is the
character of $G_{\BQ_p}$ attached to $\psi_0$ via the reciprocity
law \footnote{normalized so that geometric Frobenius elements
correspond to uniformizers} and
$\rho_\mathfrak{p}:=\rho|_{G_{F_\mathfrak{p}}}$. Under certain
conditions on $\rho$, the multiplicity of
$\widehat{\Pi}(\mathrm{rec}(\psi_0)\otimes\rho)$ in
$\widetilde{\Pi}(\rho)^{\psi_0}$ may be computed by Matsushima's
formula \cite[VII.5.2]{BoWa}.

Our knowledge on $\widetilde{\Pi}(\rho)^{\psi_0}$ and
$\widehat{\Pi}(\mathrm{rec}(\psi_0)\otimes\rho)$ is quite little.
However, using our formula in Theorem \ref{thm:main} and following
Ding's method \cite{Ding} one may give some descriptions on
$\widetilde{\Pi}(\rho)^{\psi_0,\mathrm{la}}$, the space of locally
analytic vectors in $\widetilde{\Pi}(\rho)^{\psi_0}$.

Now we focus on the case of $n=3$. Suppose that $\rho_\mathfrak{p}$
is semistable and the filtered $(\varphi,N)$-module
$\mathbf{D}_\st(\rho_\mathfrak{p})$ attached to
$\rho_{\mathfrak{p}}$ satisfies the following conditions:
\begin{eqnarray*}&& \mathbf{D}_\st(\rho_\mathfrak{p}) = Ee_1\oplus Ee_2\oplus Ee_3, \\ &&
\varphi(e_1)=\alpha_1e_1,\ \varphi(e_2)=\alpha_2 e_2, \
\varphi(e_3)= \alpha_3 e_3
\\ && Ne_1=Ne_2=0,\  Ne_3=e_1.  \end{eqnarray*} Note that
$\alpha_3=p\alpha_1$. We assume that $\alpha_2\neq p^{-1}\alpha_1,
\alpha_1,p\alpha_1, p^2\alpha_1$. There are three refinements
$\CF^{(1)}, \CF^{(2)}, \CF^{(3)}$ on
$\mathbf{D}_\st(\rho_\mathfrak{p})$:
\begin{eqnarray*}
&& \CF^{(1)}_1 = Ee_1, \ \CF^{(1)}_2 = Ee_1\oplus E e_2, \
\CF^{(1)}_3= \mathbf{D}_\st(\rho_\mathfrak{p}) ; \\
&& \CF^{(2)}_1 = Ee_2, \ \CF^{(2)}_2 = Ee_1\oplus E e_2, \
\CF^{(2)}_3= \mathbf{D}_\st(\rho_\mathfrak{p}) ; \\
&& \CF^{(3)}_1 = Ee_1, \ \CF^{(3)}_2 = Ee_1\oplus E e_3, \
\CF^{(3)}_3= \mathbf{D}_\st(\rho_\mathfrak{p}) .
\end{eqnarray*}
An interesting case is that $1$ is strongly marked for $\CF^{(1)}$.
Let $k_1$, $k_2$, $k_3$ be the ordered Hodge--Tate weights for
$\CF^{(1)}$, and suppose that $k_1<k_2<k_3$.

For each $\alpha\in E$ let $\mathrm{unr}(\alpha)$ denote the
character of $\BQ_p^\times$ whose value at $p$ is $\alpha$ and whose
restriction to $\BZ_p^\times$ is trivial. Put
$\delta_i=\mathrm{unr}(\alpha_i)x^{-k_i}$ ($i=1,2,3$). Let
$B(\BQ_p)$ be the Borel subgroup of $\GL_3(\BQ_p)$ consisting of
upper-triangular matrices. Let $\delta_B=|\cdot|^{2}\otimes 1\otimes
|\cdot|^{-2}$ be the modulus character of $B(\BQ_p)$. We consider
the locally analytic induced representation
\begin{eqnarray*}&&\hskip -20pt \Big(\mathrm{Ind}_{B(\BQ_p)}^{\GL_3(\BQ_p)}\delta_3\otimes
\delta_2 x^{-1}\otimes \delta_1 x^{-2}\Big)^{\mathrm{la}} : = \Big\{
f: \mathrm{GL}_3(\BQ_p)
\rightarrow E \ |\ f \text{ is locally analytic on }G \\
&& \hskip 15pt \text{ and }  f(bg) =\delta_B^{1/2}(b)
(\delta_3\otimes \delta_2 x^{-1}\otimes \delta_1 x^{-2})(b) f(g) \ \
\forall \   b\in B(\BQ_p) ,\ g\in \mathrm{GL}_3(\BQ_p)
\Big\}.\end{eqnarray*}  Let $\mathrm{Alg}_{k_1, k_2, k_3}$ denote
the irreducible algebraic representation of $\mathrm{GL}_3(\BQ_p)$
with the lowest weight $x^{-k_3}\otimes x^{-1-k_2}\otimes
x^{-2-k_1}$. Then there is an inclusion \begin{eqnarray*}  &&
\mathrm{Alg}_{k_1, k_2, k_3} \otimes
\Big(\mathrm{Ind}_{B(\BQ_p)}^{\GL_3(\BQ_p)}\mathrm{unr}(\alpha_3)\otimes
\mathrm{unr}(\alpha_2)\otimes
\mathrm{unr}(\alpha_1)\Big)^{\mathrm{sm}} \\
&& \hskip 100pt \hookrightarrow
\Big(\mathrm{Ind}_{B(\BQ_p)}^{\GL_3(\BQ_p)}\delta_3\otimes
\delta_2x^{-1}\otimes \delta_1x^{-2}\Big)^{\mathrm{la}},
\end{eqnarray*} where
$\Big(\mathrm{Ind}_{B(\BQ_p)}^{\GL_3(\BQ_p)}\mathrm{unr}(\alpha_3)\otimes
\mathrm{unr}(\alpha_2)\otimes
\mathrm{unr}(\alpha_1)\Big)^{\mathrm{sm}}$ is the smooth induced
representation. By the theory of Bernstein and Zelevinski \cite{BZ}
the condition $\alpha_2 \notin \{ p^{-1}\alpha_1,
\alpha_1,p\alpha_1, p^2\alpha_1 \}$ ensures that
$\Big(\mathrm{Ind}_{B(\BQ_p)}^{\GL_3(\BQ_p)}\mathrm{unr}(\alpha_3)\otimes
\mathrm{unr}(\alpha_2)\otimes
\mathrm{unr}(\alpha_1)\Big)^{\mathrm{sm}}$ has a unique irreducible
subrepresentation denoted by $S$, and its quotient by $S$ is also
irreducible.

Using Ding's method \cite[\S 4]{Ding} one may show that twisted by a
character $\widetilde{\Pi}(\rho)^{\psi_0, \mathrm{la}}$ contains an
element in
$$\mathrm{Ext}_{\mathrm{GL}_3(\BQ_p)}\Big(\mathrm{Alg}_{k_1, k_2,
k_3}\otimes S, \big(
\mathrm{Ind}_{B(\BQ_p)}^{\GL_3(\BQ_p)}\delta_3\otimes \delta_2
x^{-1}\otimes \delta_1 x^{-2} \big)^{\mathrm{la}}/
\mathrm{Alg}_{k_1, k_2, k_3}\otimes S \Big) $$ and this element only
depends on $\CL_{\CF^{(1)},1}(\mathbf{D}_\st(\rho_\mathfrak{p}))$.
For this one needs to work on some eigenvariety and use Kedlaya,
Pottharst and Xiao's result \cite[Corollary 6.3.10]{KPX} to
construct a family of Galois representations with triangulations on
the eigenvariety. The pair $(\rho, \CF^{(1)})$ corresponds to a
point on the eigenvariety. Our formula in Theorem \ref{thm:main}
should imply that the tangent space of the eigenvariety at this
point satisfies a relation. Such a relation plays an important role
in Ding's method (see the proof of \cite[Lemma 4.13]{Ding}).

There should be a similar result when either $2$ is strongly marked
for $\CF^{(2)}$ or $1$ is strongly marked for $\CF^{(3)}$.

\begin{rem} In the case when the monodromy $N$ is of rank $2$,
Breuil \cite[Th\'eor\`eme 1.2]{Breuil} gave a description on the
(conjectural) locally analytic representation attached to
$\rho_{\mathfrak{p}}$.
\end{rem}

\end{document}